%% file: ecube.tex
\documentclass[a4paper,12pt]{article}
\usepackage{graphicx}
\usepackage{amssymb}
\usepackage{amsfonts}
\usepackage{amsthm}
\usepackage{amsmath}
\usepackage{hyperref}
\usepackage{float}

\def\eqref#1{(\ref{#1})}
\hypersetup{
colorlinks=true,
linkcolor=blue,
anchorcolor=blue,
citecolor=blue
}
\newtheorem{theorem}{Theorem}[section]
\newtheorem{proposition}[theorem]{Proposition}
\newtheorem{lemma}[theorem]{Lemma}
\newtheorem{corollary}[theorem]{Corollary}
\theoremstyle{definition}
\newtheorem{definition}[theorem]{Definition}
\newtheorem{example}[theorem]{Example}

\newenvironment{problem}[1][Problem.]{\medskip\noindent \textbf{#1}\ \rm}{\smallskip}

\renewcommand{\theequation}{\thesection.\arabic{equation}}
\allowdisplaybreaks
\input{tcilatex}
\input{tcirm}

\begin{document}

\author{Alexander Grigor'yan \and Yong Lin \and S.-T. Yau \and Haohang Zhang
}
\title{Eigenvalues of the Hodge Laplacian on digraphs}
\date{May 2024}
\maketitle
\tableofcontents

\section{Introduction}

The purpose of this paper is to compute and estimate the eigenvalues of the
Hodge Laplacians on certain classes of finite digraphs (=directed graphs).
For any chain complex%
\begin{equation}
\begin{array}{cccccccc}
\dots & \overset{\partial }{\leftarrow } & \Omega _{p-1} & \overset{\partial
}{\leftarrow } & \Omega _{p} & \overset{\partial }{\leftarrow } & \Omega
_{p+1} & \overset{\partial }{\leftarrow }\dots%
\end{array}
\label{chain}
\end{equation}%
with a boundary operator $\partial $, the Hodge Laplacian $\Delta _{p}$ is
defined as an operator in $\Omega _{p}$ by
\begin{equation*}
\Delta _{p}=\partial \partial ^{\ast }+\partial ^{\ast }\partial ,
\end{equation*}%
where the adjoint operator $\partial ^{\ast }$ is defined with respect to
chosen inner products in linear spaces $\Omega _{p}$.

With any digraph $G$, one can associated a chain complex (\ref{chain}) whose
elements are certain paths in $G$ that go along arrows (see Section \ref%
{sec:review} for a detailed description). This complex that is referred to
as a \emph{path chain complex}, was introduced and investigated in a series
of papers \cite{G2022}, \cite{GLMY2013}, \cite{GLMY2020}, \cite{GLY2024},
\cite{GMY2017} etc. It reflects in a natural way the combinatorial structure
of the underlying digraph and has appropriate functorial properties with
respect to graph operations, such as Cartesian product, join, homotopy, etc.
In particular, the path chain complex satisfies the K\"{u}nneth formulas
with respect to product and join (cf. (\ref{kun}) and (\ref{kunj})).

The main results of this paper include explicit computation of the spectra
of $\Delta _{p},$ for all relevant values of $p$, for some series of
digraphs, in particular, for $n$-simplex, $n$-cube, $n$-torus and $n$-sphere
for any dimension $n$.

The digraphs $n$-simplex and $n$-sphere are defined inductively in $n$ by
using the operation \emph{join} of digraphs (see Section \ref{sec:Dmn} for
definition). As the operator $\Delta _{p}$ satisfies the product rule with
respect to join (Proposition \ref{LemDejoin}), one can use the method of
separation of variables, which in conjunction with K\"{u}nneth formula for
join allows to compute inductively the spectra of $\Delta _{p}$ on $n$%
-simplex and $n$-sphere. These results are stated in Theorem \ref{TspecDnm}.

The digraphs $n$-cube and $n$-torus are defined inductively by using
Cartesian product of digraphs (see Section \ref{SecProduct}). However, the
method of separation of variables \emph{does not work} with the canonical
inner product on spaces $\Omega _{p}$ where each path has by definition the
norm $1$, which makes the task of computing the Hodge spectra on such
digraphs much more complicated.

We have devised a new method for computing Hodge spectra in this setting
that has the following two ingredients.

(I) We have observed that the product rule does work for the so called \emph{%
normalized} Hodge operator, denoted by $\Delta _{p}^{(a)},$ where $a$ refers
to the weight that is used to redefine the inner product in the spaces $%
\Omega _{p}$. Namely, we use the weight where the norm of any path of length
$p$ is equal to
\begin{equation*}
\frac{1}{\sqrt{p!}}.
\end{equation*}%
This together with the K\"{u}nneth formula for product allows us to compute
inductively the spectra of all normalized Hodge operators $\Delta _{p}^{(a)}$
on Cartesian powers (Proposition \ref{Tdeltaa}, Theorem \ref{TspecDeaGn})
including $n$-cubes and $n$-tori (Examples \ref{ExDea}, \ref{ExDeaT}).

(II) We relate in a certain way the spectra of $\Delta _{p}$ and $\Delta
_{p}^{(a)}$ to those of operators $\mathcal{L}_{p}=\partial ^{\ast }\partial
$ also acting on $\Omega _{p}$ (Proposition \ref{prop:SpecDecmp}, Lemma \ref%
{lem:specDp''}, Corollary \ref{CorSpecDea}). Knowing the spectra of $\Delta
_{p}^{(a)}$ for all values of $p$, we compute the spectra of $\mathcal{L}%
_{p} $ and then the spectra of $\Delta _{p}.$ This program is fulfilled in
Theorems \ref{thm:specGn}, \ref{TspecDepIn}, \ref{TspecDepTn}, thus yielding
the spectra of all operators $\Delta _{p}$ on all $n$-cubes and $n$-tori.

The paper is structured as follows. In Section~\ref{sec:review} we provide
an overview of the notions of path chain complex, path homology, products of
paths and digraphs.

In Section \ref{sec:eigDecmp} we establish the relations between the spectra
of $\Delta _{p}$ and $\mathcal{L}_{p}$.

In Section \ref{sec:upperBound} we prove an upper bound for the spectrum of $%
\Delta _{1}$ in terms of combinatorial quantities (Theorem \ref{Tlamax}),
using the results of Section \ref{sec:eigDecmp}.

In Section \ref{sec:WHL} we introduce the weighted Hodge Laplacian and
obtain its general properties. In Section \ref{sec:nHodge} we define the
normalized Hodge Laplacian and apply the method of separation of variables
in abstract form.

Sections \ref{SecHodgeProduct} and \ref{sec:cube} contain the central
results of this paper. In Theorem \ref{TspecDeaGn} we obtain a rather
general formula for the spectrum of $\Delta _{p}^{(a)}$ on Cartesian powers
of digraphs, while in Theorems \ref{TspecDepIn}, \ref{TspecDepTn} we combine
this with the results of Section \ref{sec:eigDecmp} in order to compute the
spectra of $\Delta _{p}$ on $n$-cubes and $n$-tori.

Section \ref{sec:Dmn} contains the aforementioned results about spectra on
joins. In particular, Theorem \ref{TspecDnm} gives the spectra of $\Delta
_{p}$ on $n$-simplices, $n$-spheres and many other digraphs.

\textbf{Acknowledgments.} AG was funded by the Deutsche
Forschungsgemeinschaft (DFG, German Research Foundation) - Project-ID
317210226 - SFB 1283. YL was supported by NSFC, no.12071245. The work was
partially done during a visit of AG to YMSC, Tsinghua University. The
hospitality and support of YMSC is gratefully acknowledged.

\section{Path chain complex}

\label{sec:review}\setcounter{equation}{0}In this section we revise the
notion of a path chain complex on digraphs and some properties. The details
can be found in \cite{G2022}, \cite{GLMY2013}, \cite{GMY2017}, \cite%
{GLMY2020}.

\subsection{Paths and boundary operator}

Let $V$ be a finite set, whose elements are called vertices. For any $p\geq
0 $, an \textit{elementary }$p$\textit{-path} is any sequence $%
\{i_{k}\}_{k=0}^{p}$ of $p+1$ vertices in $V$, which will be denoted by $%
e_{i_{0}\dots i_{p}}=i_{0}\dots i_{p}$. Define $\Lambda _{p}=\Lambda _{p}(V)$
as a linear space of all formal linear combinations of $e_{i_{0}...i_{p}}$
with coefficients from $\mathbb{R}$, that is,
\begin{equation*}
\Lambda _{p}:=\func{span}_{\mathbb{R}}\{e_{i_{0}\cdots i_{p}}\mid
i_{0},\ldots ,i_{p}\in V\}.
\end{equation*}%
The boundary operator $\partial _{p}:\Lambda _{p}\rightarrow \Lambda _{p-1}$
is defined as
\begin{equation*}
\partial _{p}e_{i_{0}\cdots i_{p}}=\sum_{k=0}^{p}(-1)^{k}e_{i_{0}\cdots
\widehat{i}_{k}\cdots i_{p}},
\end{equation*}%
where the notation $\widehat{i}_{k}$ indicates the omission of the index $%
i_{k}$. Set also $\Lambda _{-1}=\left\{ 0\right\} $ and $\partial
_{0}e_{i}=0 $ for any $0$-path.

Normally $\partial _{p}$ will also be denoted by $\partial $ omitting the
subscript $p$. It is a fundamental property that $\partial ^{2}=0$; thus,
the pair $(\Lambda _{\ast },\partial )$ constitutes a chain complex.

An elementary path $i_{0}\dots i_{p}$ is called \textit{regular} if $%
i_{k-1}\neq i_{k}$ for all $k=1,\ldots ,p$, and as \textit{non-regular}
otherwise. The set of all regular elementary $p$-paths is denoted by $%
R_{p}=R_{p}(V)$, and the set of non-regular $p$-paths is defined as
\begin{equation*}
\mathcal{N}_{p}=\mathcal{N}_{p}(V):=\func{span}\{e_{i_{0}...i_{p}}\mid
i_{0}...i_{p}\notin R_{p}\}.
\end{equation*}%
Since $\partial \mathcal{N}_{p}\subset \mathcal{N}_{p-1}$, the boundary
operator $\partial $ is well-defined on the quotient space $\mathcal{R}%
_{p}:=\Lambda _{p}/\mathcal{N}_{p}$, thus giving rise to a chain complex $%
\left( \mathcal{R}_{\ast },\partial \right) .$ The elements of $\mathcal{R}%
_{p}$ are called \emph{regularized} $p$-paths.

A $p$-path is called regular if it is a linear combination of elementary
regular $p$-paths. It is easy to see that each regularized $p$-path as an
equivalence class contains exactly one regular $p$-path; hence, for
simplicity of notation we identify regularized $p$-paths with regular $p$%
paths.

\subsection{Path chain complex on digraph}

\label{SecChainComplex}Let $G=\left( V,E\right) $ be a digraph (=directed
graph), where $V$ is a finite set of vertices and $E$ is the set of arrows,
that is, ordered pairs $\left( i,j\right) $ where $i,j$ are distinct
vertices. An arrow in $E$ will be denoted by $i\rightarrow j$.

We say that an elementary $p$-path $i_{0}\dots i_{p}$ is \textit{allowed} if
$i_{k-1}\rightarrow i_{k}$ for all $k=1,\ldots ,p$. Let $A_{p}=A_{p}(G)$ be
the set of all allowed elementary $p$-paths. A $p$-path is called allowed if
it is a linear combination of allowed elementary $p$-paths. Denote by $%
\mathcal{A}_{p}=\mathcal{A}_{p}(G)$ the linear space of all allowed $p$%
-paths, that is,
\begin{equation*}
\mathcal{A}_{p}=\mathcal{A}_{p}(G):=\func{span}\{A_{p}\}.
\end{equation*}%
It is clear that $\mathcal{A}_{p}\subset \mathcal{R}_{p}$ but the boundary
operator $\partial $ does not have to take allowed paths to allowed.

\begin{definition}
A path $u\in \mathcal{R}_{p}$ is called $\partial $-invariant if $u\in
\mathcal{A}_{p}$ and $\partial u\in \mathcal{A}_{p-1}$.
\end{definition}

Denote by $\Omega _{p}=\Omega _{p}(G)$ the linear space of all $\partial $%
-invariant $p$-paths, that is,
\begin{equation*}
\Omega _{p}=\Omega _{p}(G):=\{u\in \mathcal{A}_{p}\mid \partial u\in
\mathcal{A}_{p-1}\}.
\end{equation*}%
It is easy to see that $\partial \Omega _{p}\subset \Omega _{p-1}$ so that
we obtain a chain complex $\left\{ \Omega _{\ast },\partial \right\} $.

\begin{definition}
The chain complex%
\begin{equation}
\begin{array}{cccccccccccc}
0 & \overset{\partial }{\leftarrow } & \Omega _{0} & \overset{\partial }{%
\leftarrow } & \Omega _{1} & \overset{\partial }{\leftarrow } & \dots &
\overset{\partial }{\leftarrow } & \Omega _{p-1} & \overset{\partial }{%
\leftarrow } & \Omega _{p} & \overset{\partial }{\leftarrow }\dots%
\end{array}%
,  \label{maing}
\end{equation}%
is called the path chain complex of the digraph $G$.
\end{definition}

Observe that
\begin{equation*}
\Omega _{0}=\mathcal{A}_{0}=\mathcal{R}_{0}=\Lambda _{0}=\limfunc{span}%
\left\{ e_{i}:i\in V\right\}
\end{equation*}%
and%
\begin{equation*}
\Omega _{1}=\mathcal{A}_{1}=\limfunc{span}\left\{ e_{ij}:i\rightarrow
j\right\} .
\end{equation*}%
In general, $\Omega _{p}$ may be a proper subspace of $\mathcal{A}_{p}$ for $%
p\geq 2$.

We say that three distinct vertices $a,b,c$ of a digraph form a \emph{%
triangle} if $i\rightarrow j\rightarrow k\ $and $i\rightarrow k$. Any
triangle determines a $\partial $-invariant $2$-path $e_{ijk}$ as $%
e_{ijk}\in \mathcal{A}_{2}$ and $\partial e_{ijk}=e_{jk}-e_{ik}+e_{ij}\in
\mathcal{A}_{1}.$ Such a path $e_{ijk}$ is also referred to as a triangle.

We say that four distinct vertices $i,j,j^{\prime },k$ of a digraph form a
\emph{square} if $i\rightarrow j\rightarrow k$, $i\rightarrow j^{\prime
}\rightarrow k$ but $i\not\rightarrow k$. Any square determines a $\partial $%
-invariant $2$-path%
\begin{equation*}
u=e_{ijk}-e_{ij^{\prime }k}
\end{equation*}%
as $u\in \mathcal{A}_{2}$ and $\partial u=e_{jk}+e_{ij}-e_{j^{\prime
}k}-e_{ij^{\prime }}\in \mathcal{A}_{1}$. The path $u$ is also called a
square.

We say that two distinct vertices $i$ and $j$ form a \emph{double arrow} if $%
i\rightarrow j\rightarrow i$. A double arrow determines $\partial $%
-invariant $2$-paths $e_{iji}$ and $e_{jij},$ that are also called double
arrows.

These three configurations are shown on Fig.~\ref{pic4}.\FRAME{ftbphFU}{%
3.6089in}{1.0672in}{0pt}{\Qcb{A triangle, a square and a double arrow}}{\Qlb{%
pic4}}{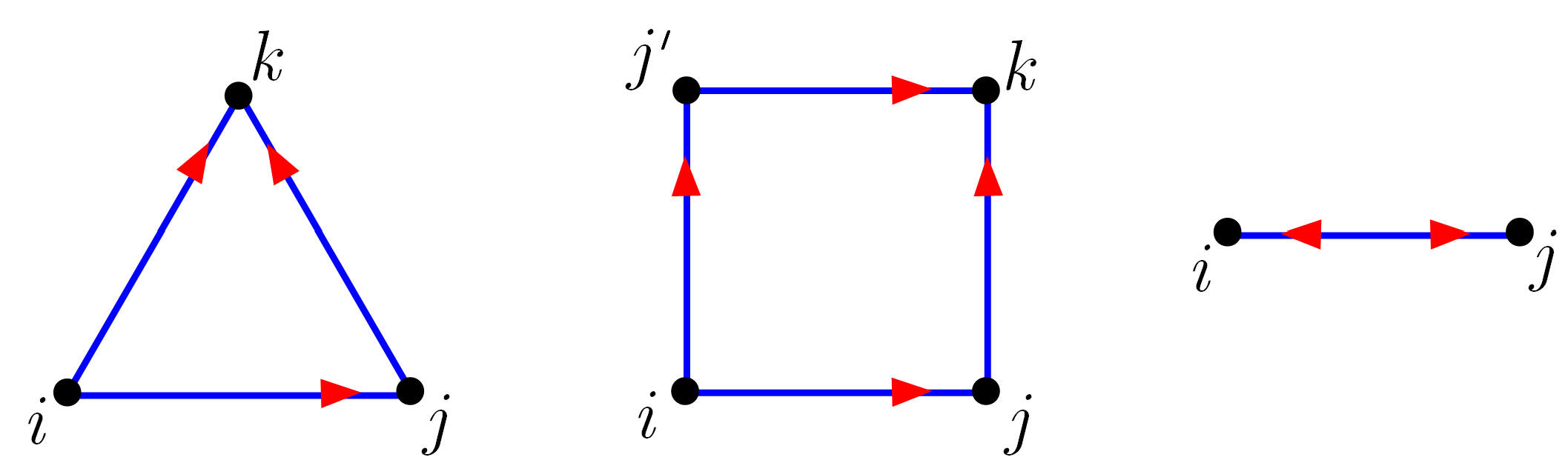}{\special{language "Scientific Word";type
"GRAPHIC";maintain-aspect-ratio TRUE;display "USEDEF";valid_file "F";width
3.6089in;height 1.0672in;depth 0pt;original-width 9.7914in;original-height
2.8746in;cropleft "0";croptop "1";cropright "1";cropbottom "0";filename
'pic4.png';file-properties "XNPEU";}}

\begin{proposition}
\label{POm2}\emph{(\cite[Proposition~2.9]{GLMY2014}, \cite[Theorem 1.8]%
{G2022})} The space $\Omega _{2}(G)$ is spanned by all double arrows,
triangles, and squares. In particular, if $G$ contains neither double nor
triangles nor squares then $\Omega _{2}(G)=\left\{ 0\right\} .$
\end{proposition}

Note that all double arrows and triangles are linearly independent, while
squares could be dependent. Indeed, consider a $m$-\emph{square}, that is, a
sequence of $m+3$ distinct vertices $\ a,\ b_{0},b_{1},...,b_{m},\ c$ such
that $a\rightarrow b_{k}\rightarrow c\ \ $for all $k=0,\dots ,m,$ while $%
a\not\rightarrow c$ (see Fig. \ref{spider}).

\FRAME{ftbphFU}{3.1903in}{1.1251in}{0pt}{\Qcb{A multisquare}}{\Qlb{spider}}{%
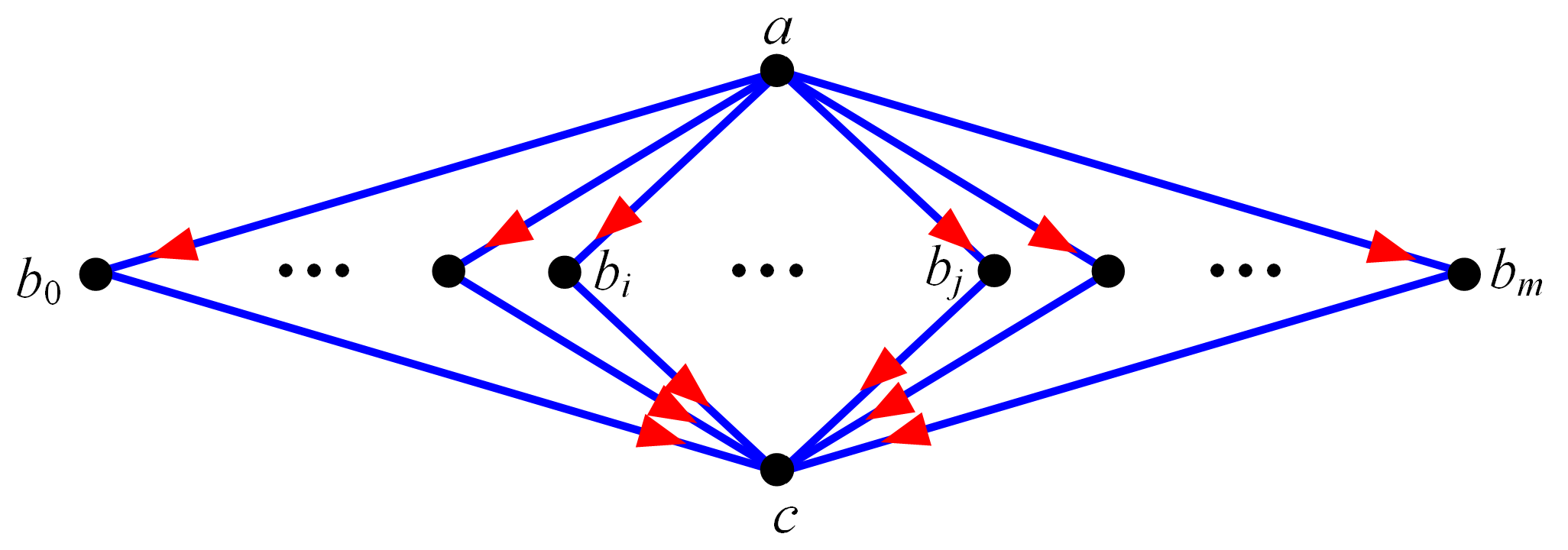}{\special{language "Scientific Word";type
"GRAPHIC";maintain-aspect-ratio TRUE;display "USEDEF";valid_file "F";width
3.1903in;height 1.1251in;depth 0pt;original-width 9.5in;original-height
3.333in;cropleft "0";croptop "1";cropright "1";cropbottom "0";filename
'spider.png';file-properties "XNPEU";}}

For example, $1$-square is a square. If $m\geq 2$ then $m$-square is also
called \emph{multisquare}. An $m$-square determines the following squares%
\begin{equation*}
u_{ij}=e_{ab_{i}c}-e_{ab_{j}c}\in \Omega _{2}\ \ \ \text{for all }%
i,j=0,...,m,
\end{equation*}%
and among them only $m$ linearly independent (for example, $\left\{
u_{0j}\right\} _{i=1}^{m})$. If $G$ contains no multisquares then all double
arrows, triangles and squares form a basis in $\Omega _{2}$.

\begin{proposition}
\label{Ppq}\emph{\cite[Proposition~3.23]{GLMY2014}} If $\Omega
_{q}(G)=\left\{ 0\right\} $ for some $q$, then $\Omega _{p}(G)=\left\{
0\right\} $ also for all $p>q$.
\end{proposition}

\subsection{Path homology}

The chain complex (\ref{maing}) determines the homology groups
\begin{equation*}
H_{p}=H_{p}(G):=\ker \partial _{p}/\func{Im}\partial _{p+1}
\end{equation*}%
that are called the \emph{path homology groups} of $G.$ The dimensions $\dim
H_{p}$ are called the Betti numbers of $G$. It is possible to prove that $%
\dim H_{0}$ is equal to the number of (undirected) connected components of $%
G $ (\cite[Proposition 1.5]{G2022}).

If the sequence $\left\{ \Omega _{p}\right\} _{p\geq 0}$ is finite, that is,
$\Omega _{q}=\left\{ 0\right\} $ for all large enough $q$, then we define
the Euler characteristic of $G$ by%
\begin{equation*}
\chi (G)=\sum_{k=0}^{\infty }\left( -1\right) ^{p+1}\dim \Omega _{p}.
\end{equation*}%
In this case, we have also%
\begin{equation*}
\chi (G)=\sum_{k=0}^{\infty }\left( -1\right) ^{p+1}\dim H_{p}.
\end{equation*}

\subsection{Product of paths and digraphs}

\label{SecProduct}Let $X,Y$ be two finite sets. Consider their product
\begin{equation*}
Z=X\times Y=\left\{ \left( x,y\right) :x\in X,y\in Y\right\} .
\end{equation*}%
Let $z=z_{0}z_{1}...z_{r}$ be a regular elementary $r$-path on $Z$, where $%
z_{k}=\left( a_{k},b_{k}\right) $ with $a_{k}\in X$ and $b_{k}\in Y$. We say
that $z$ is \emph{stair-like} if, for any $k=1,...,r$, either $a_{k-1}=a_{k}$
or $b_{k-1}=b_{k}$ is satisfied. That is, any couple $z_{k-1}z_{k}$ of
consecutive vertices is either \emph{vertical} (when $a_{k-1}=a_{k}$) or
\emph{horizontal} (when $b_{k-1}=b_{k}$).

For given elementary regular paths $x$ on $X$ and $y$ on $Y$, denote by $%
\Sigma _{x,y}$ the set of all stair-like paths $z$ on $Z$ whose projections
on $X$ and $Y$ are $x$ and $y$, respectively. For any path $z\in \Sigma
_{x,y}$ define its \emph{staircase} $S(z)$ as the image in $\mathbb{Z}%
_{+}^{2}$ as illustrated on Fig.~\ref{pic22}, and the \emph{elevation} $L(z)$
as the number of cells in $\mathbb{Z}_{+}^{2}$ below $S(z)$.

\begin{definition}
\cite[Section~3.1]{G2022} Define the \emph{cross product} of elementary
regular paths $e_{x}\in R_{\ast }(X)$ and $e_{y}\in R_{\ast }(Y)$ as a path $%
e_{x}\times e_{y}$ on $Z$ as follows:
\begin{equation*}
e_{x}\times e_{y}:=\sum_{z\in \Sigma _{x,y}}(-1)^{L(z)}e_{z},
\end{equation*}%
and extend it by linearity to all $u\in \mathcal{R}_{p}\left( X\right) $ and
$v\in \mathcal{R}_{q}\left( Y\right) $ so that $u\times v\in \mathcal{R}%
_{p+q}\left( Z\right) $.
\end{definition}

\FRAME{ftbphFU}{4.7928in}{1.6985in}{0pt}{\Qcb{A stair-like path $z\in \Sigma
_{x,y}$, its staircase $S(z)$ and its elevation $L(z)$.}}{\Qlb{pic22}}{%
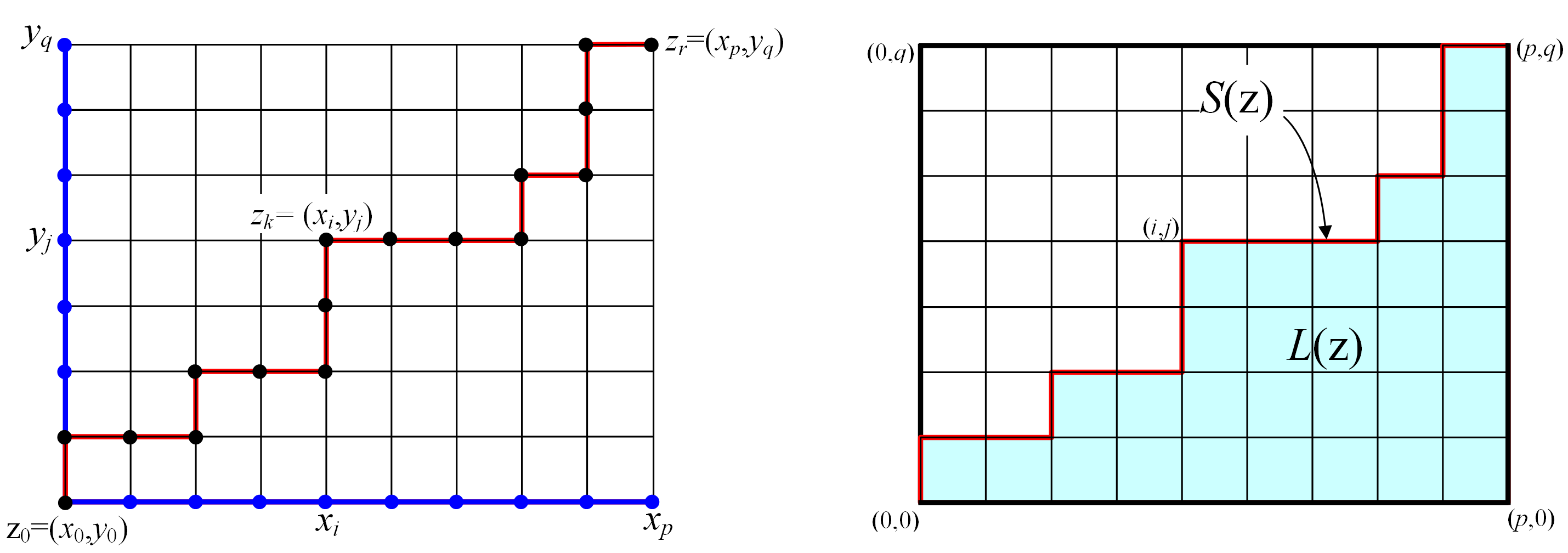}{\special{language "Scientific Word";type
"GRAPHIC";maintain-aspect-ratio TRUE;display "USEDEF";valid_file "F";width
4.7928in;height 1.6985in;depth 0pt;original-width 13.236in;original-height
4.6665in;cropleft "0";croptop "1";cropright "1";cropbottom "0";filename
'pic22.png';file-properties "XNPEU";}}

\begin{proposition}
\emph{\cite[Proposition~4.4]{GMY2017}} If $u\in \mathcal{R}_{p}(X)$ and $%
v\in \mathcal{R}_{q}(Y)$ where $p,q\geq 0$, then
\begin{equation}
\partial (u\times v)=\partial u\times v+(-1)^{p}u\times \partial v.
\label{pr}
\end{equation}
\end{proposition}

Let now $X$ and $Y$ be two digraphs (we denote their sets of vertices by the
same letters $X$ and $Y$, respectively). Define their Cartesian product $%
Z=X\square Y$ as a digraph with the set of vertices $X\times Y$, and the
arrows in $Z$ are defined as follow:%
\begin{equation*}
(x_{1},y_{1})\rightarrow (x_{2},y_{2})\ \ \Leftrightarrow \text{\ \ }%
x_{1}=x_{2}\ \text{and\ }y_{1}\rightarrow y_{2}\ \ \text{or\ \ }%
x_{1}\rightarrow x_{2}\ \text{and\ }y_{1}=y_{2},
\end{equation*}%
as is shown on the diagram:%
\begin{equation*}
\begin{array}{cccccc}
^{y_{2}}\bullet & \dots & \overset{\left( x_{1},y_{2}\right) }{\bullet } &
\rightarrow & \overset{\left( x_{2},y_{2}\right) }{\bullet } & \dots \\
\ \ \ \ \uparrow \ \  &  & \uparrow &  & \uparrow &  \\
^{y_{1}}\bullet & \dots & \overset{\left( x_{1},y_{1}\right) }{\bullet } &
\rightarrow & \overset{\left( x_{2},y_{1}\right) }{\bullet } & \dots \\
\ \ \ \ \ \ \ \ \ \  &  &  &  &  &  \\
^{Y}\ \diagup \ _{X} & \dots & \underset{x_{1}}{\bullet } & \rightarrow &
\underset{x_{2}}{\bullet }\  & \dots%
\end{array}%
\
\end{equation*}

\begin{example}
\label{ExCube}Consider an interval digraph $I=\left\{ ^{0}\bullet
\rightarrow \bullet ^{1}\right\} $ and define for any $n\geq 1$ the $n$-%
\emph{cube} as the digraph
\begin{equation*}
I^{n}=\underset{n\text{ times}}{\underbrace{I\Box ...\Box I}\ .}
\end{equation*}%
The digraphs $I^{2}$ and $I^{3}$ are shown on Fig. \ref{pic910}.\FRAME{%
ftbphFU}{2.9248in}{1.4062in}{0pt}{\Qcb{A square $I^{2}$ and a $3$-cube $%
I^{3} $ }}{\Qlb{pic910}}{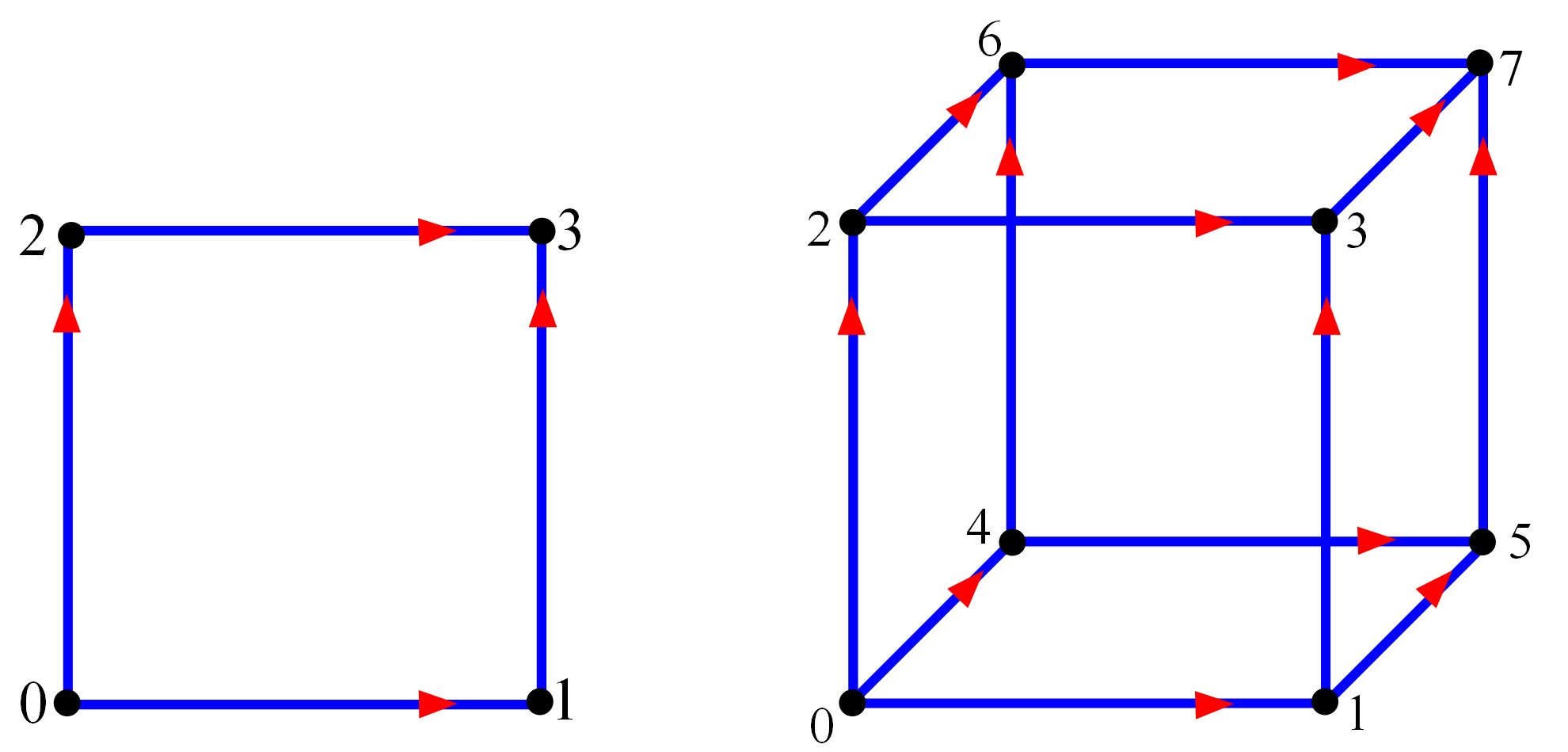}{\special{language "Scientific
Word";type "GRAPHIC";maintain-aspect-ratio TRUE;display "USEDEF";valid_file
"F";width 2.9248in;height 1.4062in;depth 0pt;original-width
9.167in;original-height 4.3932in;cropleft "0";croptop "1";cropright
"1";cropbottom "0";filename 'pic910.png';file-properties "XNPEU";}}
\end{example}

\begin{example}
\label{ExTorus}Consider a cyclic digraph $T=\left\{ ^{0}\bullet \rightarrow
\bullet ^{1}\rightarrow \bullet ^{2}\rightarrow ^{0}\right\} $ and define
for any $n\geq 1$ the $n$-\emph{torus} as the digraph
\begin{equation*}
T^{n}=\underset{n\text{\ times}}{\underbrace{T\Box ...\Box T}}.
\end{equation*}%
The tori $T$ and $T^{2}$ are shown on Fig. \ref{tori}.\FRAME{ftbphFU}{%
3.2413in}{1.3716in}{0pt}{\Qcb{A $3$-cycle $T$ and a torus $T^{2}$ embedded
on a topological torus}}{\Qlb{tori}}{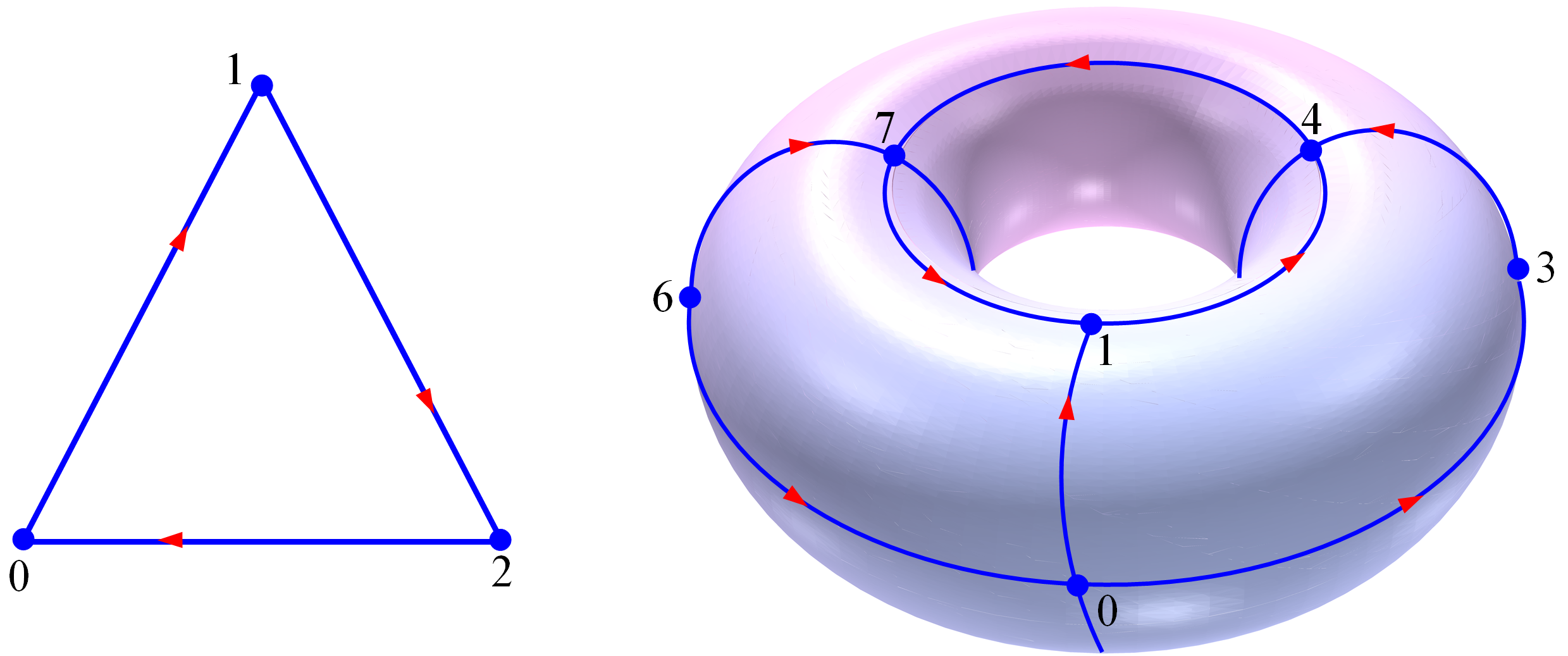}{\special{language "Scientific
Word";type "GRAPHIC";maintain-aspect-ratio TRUE;display "USEDEF";valid_file
"F";width 3.2413in;height 1.3716in;depth 0pt;original-width
11.1665in;original-height 5.393in;cropleft "0";croptop "1";cropright
"1";cropbottom "0";filename 'tori.png';file-properties "XNPEU";}}
\end{example}

It is easy to see that if $u$ and $v$ are allowed paths on $X$ and $Y$,
respectively, then $u\times v$ is allowed on $Z=X\Box Y$. It follows from (%
\ref{pr}) that if $u\in \Omega _{p}(X)$ and $v\in \Omega _{q}(Y)$ then $%
u\times v\in \Omega _{p+q}(Z).$

In fact, by the K\"{u}nneth formula of \cite[Theorem 4.7]{GMY2017}, we have,
for any $r\geq 0$,%
\begin{equation}
\Omega _{r}\left( X\Box Y\right) \cong \tbigoplus_{\left\{ p,q\geq
0:p+q=r\right\} }\left( \Omega _{p}\left( X\right) \otimes \Omega _{q}\left(
Y\right) \right) ,  \label{kun}
\end{equation}%
where the isomorphism is given by $u\otimes v\mapsto u\times v$.

The identity (\ref{kun}) implies a similar isomorphism for homology groups:%
\begin{equation}
H_{r}\left( X\Box Y\right) \cong \tbigoplus_{\left\{ p,q\geq 0:p+q=r\right\}
}\left( H_{p}\left( X\right) \otimes H_{q}\left( Y\right) \right) .
\label{kunH}
\end{equation}%
In particular, if $X$ and $Y$ are homologically trivial (that is, $%
H_{p}=\left\{ 0\right\} $ for all $p\geq 1$) then $X\Box Y$ is also
homologically trivial.

\section{Abstract Hodge Laplacian}

\label{sec:eigDecmp}\setcounter{equation}{0}In this section $\left\{ \Omega
_{p}\right\} _{p\geq 0}$ is any chain complex with a boundary operator $%
\partial :\Omega _{p}\rightarrow \Omega _{p-1}$, that is,%
\begin{equation}
\begin{array}{cccccccccccc}
0 & \leftarrow & \Omega _{0} & \overset{\partial }{\leftarrow } & \Omega _{1}
& \overset{\partial }{\leftarrow } & \dots & \overset{\partial }{\leftarrow }
& \Omega _{p-1} & \overset{\partial }{\leftarrow } & \Omega _{p} & \overset{%
\partial }{\leftarrow }\dots%
\end{array}%
,  \label{main}
\end{equation}%
where each $\Omega _{p}$ is a finite-dimensional linear space over $\mathbb{R%
}$.

Let us fix an arbitrary inner product $\left\langle \cdot ,\cdot
\right\rangle $ on each $\Omega _{p}$, which determines the adjoint boundary
operators $\partial ^{\ast }:\Omega _{p-1}\rightarrow \Omega _{p}.$ Define
the \emph{Hodge Laplacian} on $\Omega _{p}$ by
\begin{equation*}
\Delta _{p}=\partial \partial ^{\ast }+\partial ^{\ast }\partial .
\end{equation*}%
It is a non-negative definite self-adjoint operator on $\Omega _{p}$. It is
known that $0$ is the eigenvalue of $\Delta _{p}$ if and only if the
homology group $H_{p}$ of the chain complex (\ref{main}) is non-trivial;
moreover, the multiplicity of $0$ (as the eigenvalue of $\Delta _{p}$) is
equal to $\dim \,H_{p}$ (\cite[Corollary 3.4]{GLY2024}).

Denote%
\begin{equation*}
\mathcal{K}_{p}=\partial \partial ^{\ast }|_{\Omega _{p}}\ \ \text{and\ \ \ }%
\mathcal{L}_{p}=\partial ^{\ast }\partial |_{\Omega _{p}}.
\end{equation*}%
Clearly,
\begin{equation*}
\Delta _{p}=\mathcal{K}_{p}+\mathcal{L}_{p}
\end{equation*}%
and both $\mathcal{K}_{p}$ and $\mathcal{L}_{p}$ are non-negative definite
self-adjoint operators in $\Omega _{p}$ (see Fig. \ref{pic2}).

\FRAME{ftbphFU}{2.4284in}{0.7628in}{0pt}{\Qcb{Operators $\mathcal{K}_{p}$
and $\mathcal{L}_{p}$}}{\Qlb{pic2}}{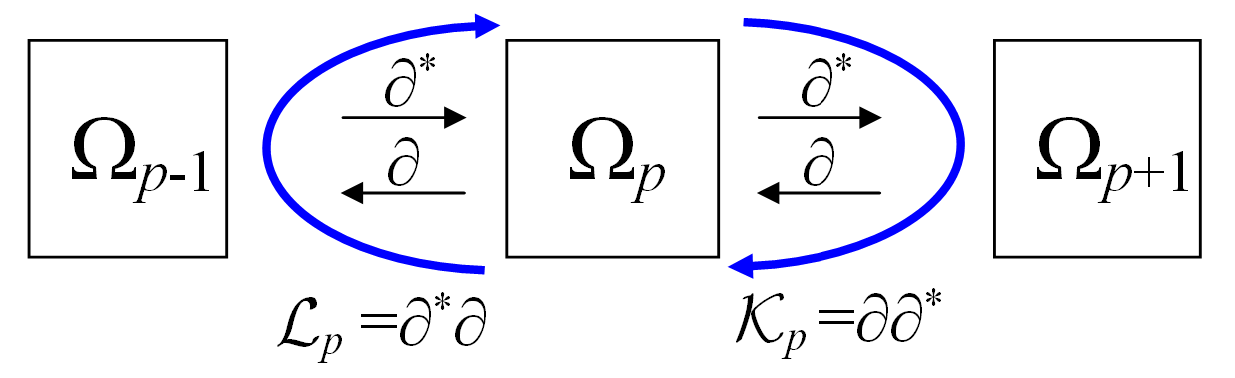}{\special{language "Scientific
Word";type "GRAPHIC";maintain-aspect-ratio TRUE;display "USEDEF";valid_file
"F";width 2.4284in;height 0.7628in;depth 0pt;original-width
6.1013in;original-height 1.9164in;cropleft "0";croptop "1";cropright
"1";cropbottom "0";filename 'pic2.png';file-properties "XNPEU";}}

Note for future references that
\begin{equation*}
\mathcal{L}_{0}=0\ \ \text{and\ \ }\Delta _{0}=\mathcal{K}_{0}.
\end{equation*}%
Fix now a positive real $\lambda $ and consider the subspaces of $\Omega
_{p} $:
\begin{equation*}
E_{p}(\lambda )=\{\varphi \in \Omega _{p}:\Delta _{p}\varphi =\lambda
\varphi \},
\end{equation*}%
\begin{align}
E_{p}^{\prime }\left( \lambda \right) & =\{\varphi \in \Omega _{p}:\mathcal{K%
}_{p}\varphi =\lambda \varphi \},  \label{E'} \\
E_{p}^{\prime \prime }(\lambda )& =\{\varphi \in \Omega _{p}:\mathcal{L}%
_{p}\varphi =\lambda \varphi \}.  \label{E''}
\end{align}

\begin{lemma}
\label{LemE+}\emph{\cite[Lemma 4.9]{GLY2024}} For any $\lambda >0$ we have%
\begin{equation}
E_{p}(\lambda )=E_{p}^{\prime }(\lambda )\oplus E_{p}^{\prime \prime
}(\lambda ).  \label{E+}
\end{equation}
\end{lemma}

\begin{proof}
Let us first verify that $E_{p}^{\prime }(\lambda )$ and $E_{\lambda
}^{\prime \prime }(\lambda )$ are subspaces of $E_{p}(\lambda ).$ Indeed, if
$\varphi \in E_{p}^{\prime }\left( \lambda \right) $ then
\begin{equation*}
\partial \partial ^{\ast }\varphi =\lambda \varphi
\end{equation*}%
whence%
\begin{equation*}
\partial \varphi =\frac{1}{\lambda }\partial \partial \partial ^{\ast
}\varphi =0,
\end{equation*}%
which implies that%
\begin{equation*}
\Delta _{p}\varphi =\partial \partial ^{\ast }\varphi +\partial ^{\ast
}\partial \varphi =\partial \partial ^{\ast }\varphi =\lambda \varphi
\end{equation*}%
and, hence, $\varphi \in E_{p}(\lambda )$. Hence, $E_{p}^{\prime }(\lambda
)\subset E_{p}(\lambda )$ and in the same way $E_{p}^{\prime \prime
}(\lambda )\subset E_{p}(\lambda ).$

Observe now that the subspace $E_{p}^{\prime }\left( \lambda \right) $ and $%
E_{p}^{\prime \prime }\left( \lambda \right) $ of $\Omega _{p}$ are
orthogonal because, for any $\varphi \in E_{p}^{\prime }\left( \lambda
\right) $ and $\psi \in E_{p}^{\prime \prime }\left( \lambda \right) ,$ we
have%
\begin{equation*}
\left\langle \varphi ,\psi \right\rangle =\frac{1}{\lambda ^{2}}\left\langle
\partial \partial ^{\ast }\varphi ,\partial ^{\ast }\partial \psi
\right\rangle =\frac{1}{\lambda ^{2}}\left\langle \partial \partial \partial
^{\ast }\varphi ,\partial \psi \right\rangle =0.
\end{equation*}%
Finally, let us prove (\ref{E+}). For any $\varphi \in E_{p}\left( \lambda
\right) $ we have%
\begin{equation*}
\left( \partial \partial ^{\ast }\right) ^{2}\varphi =\partial \partial
^{\ast }\left( \partial \partial ^{\ast }\varphi +\partial ^{\ast }\partial
\varphi \right) =\partial \partial ^{\ast }\Delta \varphi =\lambda \partial
\partial ^{\ast }\varphi ,
\end{equation*}%
which implies that $\partial \partial ^{\ast }\varphi \in E_{p}^{\prime }.$
Similarly, we have $\partial ^{\ast }\partial \varphi \in E_{p}^{\prime
\prime }.$ Hence, for any $\varphi \in E_{p}\left( \lambda \right) $ we have
\begin{equation*}
\varphi =\frac{1}{\lambda }\Delta \varphi =\frac{1}{\lambda }\partial
\partial ^{\ast }\varphi +\frac{1}{\lambda }\partial ^{\ast }\partial
\varphi ,
\end{equation*}%
whence (\ref{E+}) follows.
\end{proof}

\begin{lemma}
\label{Lemiso} For any $\lambda >0$, we have
\begin{equation*}
\dim E_{p}^{\prime }(\lambda )=\dim E_{p+1}^{\prime \prime }(\lambda )
\end{equation*}
\end{lemma}

\begin{proof}
This follows from \cite[Lemma 4.10]{GLY2024} but we give here a short
independent proof. If $\varphi \in E_{p}^{\prime }(\lambda )$ then
\begin{equation*}
\partial \partial ^{\ast }\varphi =\mathcal{K}_{p-1}\varphi =\lambda \varphi
\end{equation*}%
whence%
\begin{equation*}
\partial ^{\ast }\partial (\partial ^{\ast }\varphi )=\lambda \partial
^{\ast }\varphi .
\end{equation*}%
Since $\lambda \neq 0$, it follows that $\partial ^{\ast }\varphi \neq 0$
and, hence, $\partial ^{\ast }\varphi $ is an eigenvector of $\mathcal{L}%
_{p} $ with eigenvalue $\lambda $, that is, $\partial ^{\ast }\varphi \in
E_{p+1}^{\prime \prime }(\lambda ).$ Hence, $\partial ^{\ast }$ is a
monomorphism from $E_{p}^{\prime }(\lambda )$ to $E_{p+1}^{\prime \prime
}(\lambda )$. In the same way $\partial $ is a monomorphism from $%
E_{p+1}^{\prime \prime }(\lambda )$ to $E_{p}^{\prime }(\lambda )$, which
finishes the proof.
\end{proof}

For any finite-dimensional self-adjoint operator $A$ denote by $\func{spec}A$
the spectrum of $A$, that is, an unordered sequence of eigenvalues with
multiplicities. Observe that the problem of determination of the spectrum of
$A$ amounts to computation of dimensions of eigenspaces $\left\{ Ax=\lambda
x\right\} $ for all $\lambda \in \mathbb{R}$, that is, to multiplicities of
all real $\lambda $.

We will use the following operations with unordered sequences:

\begin{itemize}
\item disjoint union: $\left\{ \lambda _{i}\right\} \sqcup \left\{ \mu
_{j}\right\} =\left\{ \lambda _{i},\mu _{j}\right\} $ (if the same value $%
\lambda $ occurs in the both sequences, then its multiplicities add up in
the disjoint union);

\item set subtraction: $\left\{ \lambda _{i}\right\} \setminus \left\{ \mu
_{j}\right\} $ (the inverse operation to disjoint union);

\item multiplication by constant: $c\left\{ \lambda _{i}\right\} =\left\{
c\lambda _{i}\right\} $;

\item addition: $\left\{ \lambda _{i}\right\} +\left\{ \mu _{j}\right\}
=\left\{ \lambda _{i}+\mu _{j}\right\} $ (note that if one of the sequences
is empty then the sum is also empty).
\end{itemize}

For any non-negative definite self-adjoint operator $A$, set
\begin{equation*}
\func{spec}_{+}A=\func{spec}A\setminus \left\{ 0\right\} .
\end{equation*}%
The following statement provides a useful tool for computation of $\func{spec%
}\Delta _{p}$.

\begin{proposition}
\label{prop:SpecDecmp}For all $p\geq 0,$ we have%
\begin{equation}
\func{spec}_{+}\Delta _{p}=\func{spec}_{+}\mathcal{K}_{p}\sqcup \func{spec}%
_{+}\mathcal{L}_{p}  \label{spec0}
\end{equation}%
and
\begin{equation}
\func{spec}_{+}\mathcal{K}_{p}=\func{spec}_{+}\mathcal{L}_{p+1}.
\label{Dp-1Dp}
\end{equation}%
Consequently,
\begin{equation}
\func{spec}_{+}\Delta _{p}=\func{spec}_{+}\mathcal{L}_{p}\sqcup \func{spec}%
_{+}\mathcal{L}_{p+1}.  \label{spec+}
\end{equation}
\end{proposition}

\begin{proof}
Let $m_{p}(\lambda )$ be for any $\lambda >0$ the multiplicity of $\lambda $
as an eigenvalue of $\Delta _{p}$. In the same way, let $k_{p}(\lambda )$ be
the multiplicity function of $\mathcal{K}_{p}$ and $\ell _{p}(\lambda )$ --
that of $\mathcal{L}_{p}$. We have by Lemma \ref{LemE+} that, for any $%
\lambda >0,$%
\begin{equation*}
\dim E_{p}(\lambda )=\dim E_{p}^{\prime }(\lambda )+\dim E_{p}^{\prime
\prime }(\lambda ),
\end{equation*}%
that is,
\begin{equation*}
m_{p}(\lambda )=k_{p}(\lambda )+\ell _{p}(\lambda )\ \ \text{for any }%
\lambda >0.
\end{equation*}%
This identity is equivalent to (\ref{spec0}) because the disjoint union of
positive spectra is exactly addition of the multiplicities of any $\lambda
>0 $. Lemma \ref{Lemiso} yields $k_{p}(\lambda )=\ell _{p+1}(\lambda )$,
which is equivalent to (\ref{Dp-1Dp}).

Finally, (\ref{spec+}) follows from (\ref{spec0}) and (\ref{Dp-1Dp}).
Equivalently, (\ref{spec+}) can be stated as follows:%
\begin{equation}
m_{p}(\lambda )=\ell _{p}(\lambda )+\ell _{p+1}(\lambda )\ \ \text{for any }%
\lambda >0.  \label{spec++}
\end{equation}
\end{proof}

\begin{example}
\label{ExOm2=0}Assume that $\Omega _{2}=\left\{ 0\right\} .$ Since $\mathcal{%
L}_{0}=0$ and $\mathcal{L}_{2}=0$, we obtain from (\ref{spec+}) with $p=1$
and $p=0$ that%
\begin{equation*}
\func{spec}_{+}\Delta _{1}=\func{spec}_{+}\mathcal{L}_{1}=\func{spec}%
_{+}\Delta _{0}.
\end{equation*}
\end{example}

\begin{lemma}
\label{lem:specDp''}For any finite digraph $G$ and any $p\geq 1$, we have
\begin{equation}
\func{spec}_{+}\mathcal{L}_{p}=\bigsqcup_{k=0}^{\lfloor \frac{p-1}{2}\rfloor
}\func{spec}_{+}\Delta _{p-(2k+1)}\setminus \bigsqcup_{k=1}^{\lfloor \frac{p%
}{2}\rfloor }\func{spec}_{+}\Delta _{p-2k}  \label{pD''}
\end{equation}
\end{lemma}

\begin{proof}
Using the multiplicity functions from the proof of Proposition \ref%
{prop:SpecDecmp}, let us rewrite (\ref{pD''}) in an equivalent form:%
\begin{align}
\ell _{p}& =\sum_{k=0}^{\lfloor \frac{p-1}{2}\rfloor }m_{p-\left(
2k+1\right) }-\sum_{k=1}^{\lfloor \frac{p}{2}\rfloor }m_{p-2k}  \notag \\
& =\sum_{j=1}^{p}\left( -1\right) ^{j+1}m_{p-j}.  \label{pmm}
\end{align}%
Hence, it suffices to prove (\ref{pmm}). Using (\ref{spec++}) with $p$
replaced by $p-1$ and iterating, we obtain%
\begin{eqnarray*}
l_{p} &=&m_{p-1}-l_{p-1} \\
&=&m_{p-1}-(m_{p-2}-l_{p-2}) \\
&=&m_{p-1}-m_{p-2}+(m_{p-3}-l_{p-3}) \\
&&... \\
&=&m_{p-1}-m_{p-2}+m_{p-3}-...-(-1)^{p}(m_{0}-l_{0})
\end{eqnarray*}%
whence (\ref{pmm}) follows as $l_{0}\equiv 0.$
\end{proof}

\section{Upper bound for $\protect\lambda _{\max }(\Delta _{1})$}

\label{sec:upperBound}\setcounter{equation}{0}In this section we use the
techniques based on Proposition \ref{prop:SpecDecmp} in order to obtain a
certain upper bound for $\func{spec}\Delta _{1}$ on digraphs.

Here the chain complex $\left\{ \Omega _{p}\right\} _{p\geq 0}$ is the path
chain complex of a digraph $G=\left( V,E\right) $ as defined in Section \ref%
{SecChainComplex}. We endow each path space $\mathcal{R}_{p}$ with the \emph{%
canonical} inner product $\left\langle \cdot ,\cdot \right\rangle $ given by
\begin{equation}
\left\langle e_{i_{0}...i_{p}},e_{j_{0}...j_{p}}\right\rangle =\delta
_{i_{0}...i_{p}}^{j_{0}...j_{p}}.  \label{canon}
\end{equation}%
In particular, each $\Omega _{p}$ is an inner product space. The
corresponding Hodge Laplacians $\Delta _{p}$ on $\Omega _{p}$ are referred
to as \emph{canonical }Hodge Laplacians of the digraph $G$.

Given the inner product structure on $\mathcal{R}_{p}$, we introduce for
each boundary operator $\partial :\mathcal{R}_{p}\rightarrow \mathcal{R}%
_{p-1}$ the dual operator $\partial ^{\ast }:\mathcal{R}_{p-1}\rightarrow
\mathcal{R}_{p}$, defined by
\begin{equation*}
\langle \partial ^{\ast }u,v\rangle =\langle u,\partial v\rangle ,\quad
\forall u\in \mathcal{R}_{p-1},\ v\in \mathcal{R}_{p}.
\end{equation*}

For any self-adjoint operator $A$ in a finite-dimensional space, denote by $%
\lambda _{\max }(A)$ the maximal eigenvalue of $A$ and by $\lambda _{\min
}(A)$ -- the minimal eigenvalue of $A$. In this section we prove an upper
estimate of $\lambda _{\max }(\Delta _{1})$.

Consider first $\Delta _{0}.$ For any vertex $i$ of $G$, denote by $\deg (i)$
the (undirected) degree of $i$, that is, the total number of arrows having $%
i $ at the one of the ends. It was proved in \cite[Prop. 6.2]{G2022} that
\begin{equation}
\lambda _{\max }\left( \Delta _{0}\right) \leq 2\max_{i\in V}\deg i.
\label{lamax0}
\end{equation}%
Let $A$ denote the adjacency matrix of $G$, that is, $A=\left( a_{ij}\right)
_{i,j\in V}$ where
\begin{equation*}
a_{ij}=\left\{
\begin{array}{ll}
1, & \text{if }i\rightarrow j\text{ or }j\rightarrow i \\
0 & \text{otherwise.}%
\end{array}%
\right.
\end{equation*}%
If $G$ contains no double arrow then, in the basis $\left\{ e_{i}\right\} ,$
\begin{equation}
\text{matrix\ of\ }\Delta _{0}=\func{diag}\left( \deg (i)\right) -A
\label{made0}
\end{equation}%
(see \cite[Example 6.9]{G2022}). In particular, $\Delta _{0}$ does not
depend on the choice of orientation of arrows.

For any arrow $\xi $ in $G$, denote by $\deg _{\Delta }\xi $ the number of
triangles containing the arrow $\xi $, and by $\deg _{\Box }\xi $ -- the
number of squares containing $\xi $. Recall that the notion of multisquare
was defined in Section \ref{SecChainComplex} (see also Fig. \ref{spider}).

\begin{theorem}
\label{Tlamax}Assume that digraph $G$ contains neither multisquares nor
double arrows. Then
\begin{equation}
\lambda _{\max }\left( \Delta _{1}\right) \leq \max \left( 2\max_{i\in
V}\deg i,~\max_{\xi \in E}\left( 3\deg _{\Delta }\xi +2\deg _{\Box }\xi
\right) \right) .  \label{lamax1<}
\end{equation}
\end{theorem}

The estimate (\ref{lamax1<}) improves significantly the estimate
\begin{equation}
\lambda _{\max }\left( \Delta _{1}\right) \leq 2\max_{i\in V}\deg
i+3\max_{\xi \in E}\deg _{\Delta }\xi +2\max_{\xi \in E}\deg _{\Box }\xi
\label{lamax}
\end{equation}%
that was proved in \cite[Theorem 6.20]{G2022}.

\begin{proof}
By (\ref{spec0}) and (\ref{Dp-1Dp}) we have%
\begin{equation*}
\func{spec}_{+}\Delta _{0}=\func{spec}_{+}\mathcal{K}_{0}
\end{equation*}%
and
\begin{equation*}
\func{spec}_{+}\Delta _{1}=\func{spec}_{+}\mathcal{K}_{0}\sqcup \func{spec}%
_{+}\mathcal{K}_{1}
\end{equation*}%
whence%
\begin{equation}
\func{spec}_{+}\Delta _{1}=\func{spec}_{+}\Delta _{0}\sqcup \func{spec}_{+}%
\mathcal{K}_{1}.  \label{specDe1K1}
\end{equation}%
Hence, the estimate (\ref{lamax1<}) will follow from (\ref{specDe1K1}), (\ref%
{lamax0}) and
\begin{equation}
\lambda _{\max }(\mathcal{K}_{1})\leq ~\max_{\xi \in E}\left( 3\deg _{\Delta
}\xi +2\deg _{\Box }\xi \right) .  \label{lamaxD''}
\end{equation}%
Let us now prove (\ref{lamaxD''}). Since the Rayleigh quotient of $\mathcal{K%
}_{1}$ is%
\begin{equation*}
\frac{\left\langle \mathcal{K}_{1}u,u\right\rangle }{\left\Vert u\right\Vert
^{2}}=\frac{\left\langle \partial \partial ^{\ast }u,u\right\rangle }{%
\left\Vert u\right\Vert ^{2}}=\frac{\left\Vert \partial ^{\ast }u\right\Vert
^{2}}{\left\Vert u\right\Vert ^{2}},
\end{equation*}%
it suffices to prove that, for any $u\in \Omega _{1}$,
\begin{equation}
\left\Vert \partial ^{\ast }u\right\Vert ^{2}\leq \max_{\xi \in E}\left(
3\deg _{\Delta }\xi +2\deg _{\Box }\xi \right) \left\Vert u\right\Vert ^{2}.
\label{d*u1}
\end{equation}%
Let $u=\sum_{\xi \in E}u^{\xi }e_{\xi }$ and, hence,%
\begin{equation*}
\left\Vert u\right\Vert ^{2}=\sum_{\xi \in E}(u^{\xi })^{2}.
\end{equation*}%
By the hypothesis, all triangles and squares form an orthogonal basis in $%
\Omega _{2}$, denote it by $\left\{ \gamma _{n}\right\} $. Using this basis
in $\Omega _{2}$, we have%
\begin{equation*}
\left\Vert \partial ^{\ast }u\right\Vert ^{2}=\sum_{n}\frac{\left\langle
\partial ^{\ast }u,\gamma _{n}\right\rangle ^{2}}{\left\Vert \gamma
_{n}\right\Vert ^{2}}=\sum_{n}\frac{\left\langle u,\partial \gamma
_{n}\right\rangle ^{2}}{\left\Vert \gamma _{n}\right\Vert ^{2}}.
\end{equation*}%
If $\gamma _{n}$ is a triangle $e_{abc}$ then $\left\Vert \gamma
_{n}\right\Vert =1$ and%
\begin{equation*}
\left\langle u,\partial \gamma _{n}\right\rangle =\left\langle
u,e_{bc}-e_{ac}+e_{ab}\right\rangle =u^{bc}-u^{ac}+u^{ab},
\end{equation*}%
\begin{equation}
\frac{\left\langle u,\partial \gamma _{n}\right\rangle ^{2}}{\left\Vert
\gamma _{n}\right\Vert ^{2}}\leq 3\left(
(u^{bc})^{2}+(u^{ac})^{2}+(u^{ab})^{2}\right) .  \label{ga3}
\end{equation}%
If $\gamma _{n}$ is a square $e_{abc}-e_{ab^{\prime }c}$ then $\left\Vert
\gamma _{n}\right\Vert ^{2}=2$ and%
\begin{equation*}
\left\langle u,\partial \gamma _{n}\right\rangle =\left\langle
u,e_{ab}+e_{bc}-e_{ab^{\prime }}+e_{b^{\prime }c}\right\rangle
=u^{ab}+u^{bc}-u^{ab^{\prime }}-u^{b^{\prime }c},
\end{equation*}%
\begin{equation}
\frac{\left\langle u,\partial \gamma _{n}\right\rangle ^{2}}{\left\Vert
\gamma _{n}\right\Vert ^{2}}\leq 2\left(
(u^{ab})^{2}+(u^{bc})^{2}+(u^{ab^{\prime }})^{2}+(u^{b^{\prime
}c})^{2}\right) .  \label{ga2}
\end{equation}%
Let us sum up (\ref{ga3}) over all triangles $\gamma _{n}$ and (\ref{ga2})
over all squares $\gamma _{n}$ and observe that, for each arrow $\xi \in E$,
the component $\left( u^{\xi }\right) ^{2}$ appears in the sum (\ref{ga3})
for $\deg _{\Delta }\xi $ triangles $\gamma _{n}$ and in the sum (\ref{ga2})
for $\deg _{\Box }\xi $ squares $\gamma _{n}$. Hence, we obtain that%
\begin{eqnarray*}
\sum_{n}\frac{\left\langle u,\partial \gamma _{n}\right\rangle ^{2}}{%
\left\Vert \gamma _{n}\right\Vert ^{2}} &\leq &\sum_{\xi \in E}\left( 3\deg
_{\Delta }\xi +2\deg _{\Box }\xi \right) \left( u^{\xi }\right) ^{2} \\
&\leq &\max_{\xi \in E}\left( 3\deg _{\Delta }\xi +2\deg _{\Box }\xi \right)
\sum_{\xi \in E}\left( u^{\xi }\right) ^{2},
\end{eqnarray*}%
whence (\ref{d*u1}) follows.
\end{proof}

\begin{corollary}
Under the hypothesis of Theorem \emph{\ref{Tlamax}}, if in addition $\deg
_{\Delta }\xi \leq 2$ for all $\xi \in E$ then
\begin{equation}
\lambda _{\max }(\Delta _{1})\leq 2\max_{i\in V}\deg i.  \label{la1maxi}
\end{equation}
\end{corollary}

\begin{proof}
Let $\xi =ij\in E$ be an arrow where $\max_{\xi \in E}\left( 3\deg _{\Delta
}\xi +2\deg _{\Box }\xi \right) $ is attained. The vertex $i$ has adjacent
arrow $\xi $; besides, each of $\deg _{\Delta }\xi $ triangles attached to $%
\xi $ gives one more adjacent arrow to $i$, and so do all $\deg _{\Box }\xi $
squares attached to $\xi $. Hence, we obtain%
\begin{equation*}
\deg i\geq 1+\deg _{\Delta }\xi +\deg _{\Box }\xi .
\end{equation*}%
Using $\deg _{\Delta }\xi \leq 2$, we obtain%
\begin{equation*}
2\deg i\geq 2+2\deg _{\Delta }\xi +2\deg \Box \xi \geq 3\deg _{\Delta }\xi
+2\deg \Box \xi .
\end{equation*}%
Hence, (\ref{la1maxi}) follows from (\ref{lamax1<}).
\end{proof}

\begin{example}
Let $I$ be the interval digraph, that is, $I=\left\{ ^{0}\bullet \rightarrow
\bullet ^{1}\right\} $. As in Example \ref{ExCube}, consider the $n$-cube
\begin{equation*}
I^{n}=\underset{n\text{ times}}{\underbrace{I\Box ...\Box I}}.
\end{equation*}%
It is easy to see that $I^{n}$ contains no triangles. Since $\deg \left(
i\right) =n$ for any vertex $i$, it follows from (\ref{la1maxi}) that, for
the digraph $I^{n}$,
\begin{equation*}
\lambda _{\max }\left( \Delta _{1}\right) \leq 2n.
\end{equation*}%
As we will see in Section~\ref{specDepIn}, for $I^{n}$ we have, in fact, the
equality $\lambda _{\max }(\Delta _{1})=2n$.
\end{example}

In fact, In all known examples (\ref{la1maxi}) is satisfied even if $\deg
_{\Delta }\xi >2$ for some $\xi $.

\section{Weighted Hodge Laplacian}

\label{sec:WHL}\setcounter{equation}{0}As in Section \ref{sec:upperBound},
we denote by $\Delta _{p}$ the canonical Hodge Laplacian on a digraph $G$
associated with the canonical inner product (\ref{canon}). As in section \ref%
{sec:eigDecmp}, we have
\begin{equation*}
\Delta _{p}=\mathcal{K}_{p}+\mathcal{L}_{p}
\end{equation*}%
where
\begin{equation*}
\mathcal{K}_{p}=\partial \partial ^{\ast }|_{\Omega _{p}}\ \ \text{and\ \ \ }%
\mathcal{L}_{p}=\partial ^{\ast }\partial |_{\Omega _{p}}
\end{equation*}%
and $\partial ^{\ast }$ is the adjoint to $\partial $ with respect to the
canonical inner product.

Fix a sequence of positive numbers $a=\left\{ a_{p}\right\} _{p=0}^{\infty }$
and define another another inner product in $\mathcal{R}_{p}(G)$ that will
be referred to as \emph{weighted inner product}:%
\begin{equation*}
\left\langle e_{i_{0}...i_{p}},e_{j_{0}...j_{p}}\right\rangle _{a}:=\frac{1}{%
a_{p}}\left\langle e_{i_{0}...i_{p}},e_{j_{0}...j_{p}}\right\rangle =\frac{1%
}{a_{p}}\delta _{i_{0}...i_{p}}^{j_{0}...j_{p}}.
\end{equation*}%
In particular, we have $\left\Vert e_{i_{0}...i_{p}}\right\Vert _{a}^{2}=%
\frac{1}{a_{p}}.$ Denote by $\partial _{a}^{\ast }$ the adjoint operator of $%
\partial $ with respect to the weighted inner product, and by $\Delta
_{p}^{\left( a\right) }$ the corresponding Hodge Laplacian of $\left\{
\Omega _{p}\right\} _{p\geq 0}$:%
\begin{equation*}
\Delta _{p}^{(a)}=\partial \partial _{a}^{\ast }+\partial _{a}^{\ast
}\partial =\mathcal{K}_{p}^{\left( a\right) }+\mathcal{L}_{p}^{\left(
a\right) }
\end{equation*}%
where
\begin{equation*}
\mathcal{K}_{p}^{(a)}=\partial \partial _{a}^{\ast }|_{\Omega _{p}}\ \ \text{%
and\ \ }\mathcal{L}_{p}^{\left( a\right) }=\partial _{a}^{\ast }\partial
|_{\Omega _{p}}
\end{equation*}%
For convenience of notation, let us set $a_{-1}=1.$

\begin{lemma}
\label{Lemap}We have for $p\geq 0$%
\begin{equation}
\mathcal{K}_{p}^{(a)}=\frac{a_{p+1}}{a_{p}}\mathcal{K}_{p}  \label{D'}
\end{equation}%
and%
\begin{equation}
\mathcal{L}_{p}^{\left( a\right) }=\frac{a_{p}}{a_{p-1}}\mathcal{L}_{p}.
\label{D''}
\end{equation}%
Consequently,
\begin{equation}
\Delta _{p}^{\left( a\right) }=\frac{a_{p+1}}{a_{p}}\mathcal{K}_{p}+\frac{%
a_{p}}{a_{p-1}}\mathcal{L}_{p}.  \label{De12}
\end{equation}
\end{lemma}

In particular, for $p=0$ we have $\mathcal{L}_{0}=0$ and, hence,
\begin{equation}
\Delta _{0}^{\left( a\right) }=\frac{a_{1}}{a_{0}}\mathcal{K}_{0}=\frac{a_{1}%
}{a_{0}}\Delta _{0}  \label{De0}
\end{equation}

\begin{proof}
We have, for $u\in \Omega _{p}$ and $\omega \in \Omega _{p+1},$%
\begin{equation*}
\left\langle \partial _{a}^{\ast }u,\omega \right\rangle _{a}=\left\langle
u,\partial \omega \right\rangle _{a}=\frac{1}{a_{p}}\left\langle u,\partial
\omega \right\rangle =\frac{1}{a_{p}}\left\langle \partial ^{\ast }u,\omega
\right\rangle =\frac{a_{p+1}}{a_{p}}\left\langle \partial ^{\ast }u,\omega
\right\rangle _{a}
\end{equation*}%
whence%
\begin{equation*}
\partial _{a}^{\ast }u=\frac{a_{p+1}}{a_{p}}\partial ^{\ast }u\ \ \text{for
all }u\in \Omega _{p}.
\end{equation*}%
It follows that, for $u\in \Omega _{p}$,%
\begin{equation*}
\mathcal{K}_{p}^{(a)}u=\partial _{a}\partial ^{\ast }u=\frac{a_{p+1}}{a_{p}}%
\partial \partial ^{\ast }u=\frac{a_{p+1}}{a_{p}}\mathcal{K}_{p}u,
\end{equation*}%
which proves (\ref{D'}).

If $p=0$ then (\ref{D''}) is trivial as $\mathcal{L}_{p}=0.$ If $p\geq 1$
then we have
\begin{equation*}
\mathcal{L}_{p}^{\left( a\right) }u=\partial _{a}^{\ast }\partial u=\frac{%
a_{p}}{a_{p-1}}\partial ^{\ast }\partial u=\frac{a_{p}}{a_{p-1}}\partial
^{\ast }\partial u=\frac{a_{p}}{a_{p-1}}\mathcal{L}_{p}u,
\end{equation*}%
which proves (\ref{D''}) and, hence, (\ref{De12}).
\end{proof}

\begin{corollary}
\label{CorSpecDea}We have for all $p\geq 0$%
\begin{equation}
\func{spec}_{+}\Delta _{p}^{\left( a\right) }=\frac{a_{p+1}}{a_{p}}\func{spec%
}_{+}\mathcal{K}_{p}\,\sqcup \,\frac{a_{p}}{a_{p-1}}\func{spec}_{+}\mathcal{L%
}_{p}  \label{specDea}
\end{equation}%
and
\begin{equation}
\func{spec}_{+}\Delta _{p}^{\left( a\right) }=\frac{a_{p}}{a_{p-1}}\func{spec%
}_{+}\mathcal{L}_{p}\,\sqcup \,\frac{a_{p+1}}{a_{p}}\func{spec}_{+}\mathcal{L%
}_{p+1}.  \label{specDea+}
\end{equation}
\end{corollary}

\begin{proof}
Indeed, applying (\ref{spec0}) and (\ref{spec+}) to operator $\Delta
_{p}^{\left( a\right) }$, we obtain%
\begin{equation*}
\func{spec}_{+}\Delta _{p}^{\left( a\right) }=\func{spec}_{+}\mathcal{K}%
_{p}^{(a)}\sqcup \func{spec}_{+}\mathcal{L}_{p}^{\left( a\right) }
\end{equation*}%
and
\begin{equation*}
\func{spec}_{+}\Delta _{p}^{\left( a\right) }=\func{spec}_{+}\mathcal{L}%
_{p}^{\left( a\right) }\sqcup \func{spec}_{+}\mathcal{L}_{p+1}^{\left(
a\right) }.
\end{equation*}%
which together with Lemma \ref{Lemap} yields (\ref{specDea}) and (\ref%
{specDea+}).
\end{proof}

\begin{proposition}
\label{Pspeca}For any $p\geq 0$, $\func{spec}\Delta _{p}^{(a)}$ is
determined by the sequence $\left\{ \func{spec}\Delta _{q}\right\}
_{q=0}^{p} $ and the weight $a$. Conversely, $\func{spec}\Delta _{p}$ is
determined by the sequence $\left\{ \func{spec}\Delta _{q}^{(a)}\right\}
_{q=0}^{p}$ and $a $.
\end{proposition}

\begin{proof}
By (\ref{specDea+}), $\func{spec}_{+}\Delta _{p}^{(a)}$ is determined by $%
\func{spec}_{+}\mathcal{L}_{p}$ and $\func{spec}_{+}\mathcal{L}_{p+1}$
(together with $a$). By Lemma \ref{lem:specDp''}, $\func{spec}_{+}\mathcal{L}%
_{p}$ is determined by $\left\{ \func{spec}_{+}\Delta _{q}\right\}
_{q=0}^{p-1}$, and similarly $\func{spec}_{+}\mathcal{L}_{p+1}$ is
determined by $\left\{ \func{spec}_{+}\Delta _{q}\right\} _{q=0}^{p}$.
Hence, $\func{spec}_{+}\Delta _{p}^{(a)}$ is determined by $\left\{ \func{%
spec}_{+}\Delta _{q}\right\} _{q=0}^{p}.$

It remains to handle $0$ as an eigenvalue of $\Delta _{p}^{(a)}$. The
multiplicity of $0$ is equal to $\dim H_{p}$ and is independent of the
weight $a$. Hence, the full spectrum $\func{spec}\Delta _{p}^{(a)}$ is
determined by the sequence $\left\{ \func{spec}\Delta _{q}\right\}
_{q=0}^{p}.$

The second claim is proved in the same way.
\end{proof}

\begin{corollary}
\label{Corspeca}Assume that $\Omega _{r}=\left\{ 0\right\} $ for all $r>p$.
Then $\func{spec}\Delta _{p}^{(a)}$ is determined by the sequence $\left\{
\func{spec}\Delta _{q}\right\} _{q=0}^{p-1}$, the Euler characteristic $\chi
(G)$ and the weight $a$.
\end{corollary}

\begin{proof}
Since $\Omega _{p+1}=\left\{ 0\right\} $, by the argument from the previous
proof $\func{spec}_{+}\Delta _{p}^{(a)}$ is determined by $\func{spec}_{+}%
\mathcal{L}_{p}$ and, hence, by $\left\{ \func{spec}_{+}\Delta _{q}\right\}
_{q=0}^{p-1}.$ It remains to observe that the multiplicity of $0$ as an
eigenvalue of $\Delta _{p}^{(a)}$, that is, $\dim H_{p}$, is determined by $%
\left\{ \func{spec}\Delta _{q}\right\} _{q=0}^{p-1}$ and $\chi $, which
follows from
\begin{equation*}
\chi =\dim H_{0}-\dim H_{1}+...+\left( -1\right) ^{p}\dim H_{p-1}+\left(
-1\right) ^{p+1}\dim H_{p}.
\end{equation*}
\end{proof}

\section{Normalized Hodge Laplacian}

\label{sec:nHodge}\setcounter{equation}{0}Consider now the following
specific weight
\begin{equation}
a_{p}=p!.  \label{ap}
\end{equation}

\begin{definition}
The inner product $\left\langle \cdot ,\cdot \right\rangle _{a}$ with the
weight (\ref{ap}) will be referred to as the \emph{normalized} inner
product, and the corresponding weighted Hodge Laplacian will be called the
\emph{normalized} Hodge Laplacian.
\end{definition}

For example, for the normalized Hodge Laplacian from (\ref{specDea}) and (%
\ref{specDea+}) that%
\begin{equation}
\func{spec}_{+}\Delta _{p}^{\left( a\right) }=\left( p+1\right) \func{spec}%
_{+}\mathcal{K}_{p}\,\sqcup \,p\func{spec}_{+}\mathcal{L}_{p}
\label{specDen}
\end{equation}%
and
\begin{equation}
\func{spec}_{+}\Delta _{p}^{\left( a\right) }=p\func{spec}_{+}\mathcal{L}%
_{p}\,\sqcup \,(p+1)\func{spec}_{+}\mathcal{L}_{p+1}.  \label{specDen+}
\end{equation}%
Besides, it follows from (\ref{De0}) that%
\begin{equation}
\Delta _{0}^{\left( a\right) }=\Delta _{0}  \label{De0n}
\end{equation}

The significance of the weight (\ref{ap}) is determined by the following
statement, where we use the Cartesian product $X\Box Y$ of two digraphs.

\begin{proposition}
\emph{\cite[Lemma 4.7]{GLY2024}} Let $a$ be the weight \emph{(\ref{ap})}.
Let $X$ and $Y$ be two digraphs. Then, for $u\in \Omega _{p}\left( X\right) $%
, $v\in \Omega _{q}\left( Y\right) $ (where $p,q\geq 0$) and $r=p+q$,%
\begin{equation}
\Delta _{r}^{\left( a\right) }\left( u\times v\right) =\Delta _{p}^{\left(
a\right) }u\times v+u\times \Delta _{q}^{\left( a\right) }v.
\label{aproduct}
\end{equation}
\end{proposition}

Note that the proof of (\ref{aproduct}) in \cite{GLY2024} is based on the
following ingredients:

\begin{enumerate}
\item the product rule (\ref{pr}) for $\partial $;

\item the K\"{u}nneth formula (\ref{kun});

\item if $\ u\in \mathcal{A}_{p}\left( X\right) $, $v\in \mathcal{A}%
_{q}\left( Y\right) $, $\varphi \in \mathcal{A}_{p^{\prime }}\left( X\right)
$ and $\psi \in \mathcal{A}_{q^{\prime }}\left( Y\right) $ then%
\begin{equation*}
\langle u\times v,\varphi \times \psi \rangle =\tbinom{p+q}{p}\langle
u,\varphi \rangle \langle v,\psi \rangle ,
\end{equation*}%
which implies for the weight (\ref{ap}) that
\begin{equation*}
\langle u\times v,\varphi \times \psi \rangle _{a}=\langle u,\varphi \rangle
_{a}\langle v,\psi \rangle _{a}.
\end{equation*}
\end{enumerate}

The following statement is a combination of the argument of separation of
variables, based on (\ref{aproduct}), and the K\"{u}nneth formula (\ref{kun}%
).

\begin{proposition}
\label{Tdeltaa}For the weight \emph{(\ref{ap})} we have, for any $r\geq 0$,%
\begin{equation}
\func{spec}\Delta _{r}^{\left( a\right) }\left( X\Box Y\right)
=\tbigsqcup\limits_{\left\{ p,q\geq 0:\,p+q=r\right\} }\left( \func{spec}%
\Delta _{p}^{\left( a\right) }\left( X\right) +\func{spec}\Delta
_{q}^{\left( a\right) }\left( Y\right) \right) .  \label{Der}
\end{equation}
\end{proposition}

\begin{proof}
Observe that if $u\in \Omega _{p}\left( X\right) $ and $v\in \Omega
_{q}\left( Y\right) $ are eigenvectors such that
\begin{equation*}
\Delta _{p}^{(a)}u=\lambda u\ \ \text{and\ \ }\Delta _{q}^{(a)}v=\mu v,
\end{equation*}%
then we have by (\ref{aproduct}) for $r=p+q$:
\begin{equation*}
\Delta _{r}^{(a)}\left( u\times v\right) =(\Delta _{p}^{(a)}u)\ast v+u\ast
\Delta _{q}^{(a)}v=\left( \lambda +\mu \right) \left( u\times v\right) ,
\end{equation*}%
that is, $u\times v$ is an eigenvector of $\Delta _{r}^{(a)}$ on $X\Box Y$
with the eigenvalue $\lambda +\mu $.

In each $\Omega _{p}\left( X\right) $ there is a basis that consists of
eigenvectors of $\Delta _{p}^{(a)}$; denote by $\left\{ u_{k}\right\} $ the
union of all such bases of $\Omega _{p}\left( X\right) $ across all $p\geq 0$%
, with the corresponding eigenvalues $\left\{ \lambda _{k}\right\} $. Let $%
\left\{ v_{l}\right\} $ be a similar sequence on $Y$ with the eigenvalues $%
\left\{ \mu _{l}\right\} .$ By the K\"{u}nneth formula (\ref{kun}), we have,
for any $r\geq 0$,%
\begin{equation}
\Omega _{r}\left( X\Box Y\right) \cong \tbigoplus_{\left\{ p,q\geq
0:p+q=r\right\} }\left( \Omega _{p}\left( X\right) \otimes \Omega _{q}\left(
Y\right) \right) ,  \label{KunOm}
\end{equation}%
where the isomorphism is given by $u\otimes v\mapsto u\times v$. It follows
that $\Omega _{r}\left( X\Box Y\right) $ has a basis
\begin{equation*}
\left\{ u_{k}\times v_{l}:\left\vert u_{k}\right\vert +\left\vert
v_{l}\right\vert =r\right\} ,
\end{equation*}%
where $\left\vert \cdot \right\vert $ denotes here the length of paths. The
elements of this basis are the eigenvectors of $\Delta _{r}^{(a)}$ on $X\Box
Y$ with eigenvalues $\lambda _{k}+\mu _{l}$, whence (\ref{specjoin}) follows.
\end{proof}

\section{Hodge Laplacian on Cartesian products}

\label{SecHodgeProduct}\setcounter{equation}{0}

\subsection{Spectrum of the normalized Hodge Laplacian on $G^{n}$}

\label{sec:Gn}Let $G=\left( V,E\right) $ be any finite connected digraph and
$a$ be the weight (\ref{ap}). Assuming that $\func{spec}\Delta _{p}^{(a)}(G)$
is known for all $p\geq 0$, the spectrum $\func{spec}\Delta _{p}^{\left(
a\right) }(G^{n})$ can be determined by induction using Proposition~\ref%
{Tdeltaa}. We provide here an explicit expression for $\func{spec}\Delta
_{p}^{(a)}(G^{n})$ in the case when the spaces $\Omega _{p}(G)$ are trivial
for $p>1$. It follows from the K\"{u}nneth formula \ref{KunOm} that $\Omega
_{r}(G^{n})=\left\{ 0\right\} $ for $r>n$.

In Theorem \ref{TspecDeaGn} below we obtain formulas for $\func{spec}\Delta
_{r}^{\left( a\right) }(G^{n})$ for all $0\leq r\leq n$ using $\func{spec}%
\Delta _{0}(G)$ only. We start with a lemma.

\begin{lemma}
Assume that $\Omega _{p}(G)=\left\{ 0\right\} $ for all $p\geq 2.$ Then, for
all $0\leq r\leq n$, we have
\begin{align}
\func{spec}\Delta _{r}^{(a)}(G^{n})& =\left\{ (\alpha _{1}+\cdots +\alpha
_{n-r}+\beta _{1}+\cdots +\beta _{r})_{\binom{n}{r}}\,\right\vert  \notag \\
& \ \ \ \ \ \ \ \left. \alpha _{i}\in \func{spec}\Delta
_{0}^{(a)}(G),\,\beta _{j}\in \func{spec}\Delta _{1}^{(a)}(G)\}\right\} .
\label{ab}
\end{align}
\end{lemma}

\begin{proof}
This identity is trivial for $n=1$. For the induction step from $n-1$ to $n$%
, we obtain, using (\ref{Der}) and the induction hypothesis, that
\begin{align*}
\func{spec}\Delta _{r}^{(a)}(G^{n})& =\tbigsqcup\limits_{\left\{ p,q\geq
0:\,p+q=r\right\} }\left( \func{spec}\Delta _{p}^{\left( a\right) }(G^{n-1})+%
\func{spec}\Delta _{q}^{\left( a\right) }(G)\right) \\
& =\left( \func{spec}\Delta _{r}^{(a)}(G^{n-1})+\func{spec}\Delta
_{0}^{(a)}(G)\right) \ \ \ \ \ \ \ \ \ \ \ \ \ \ \ (p=r,q=0) \\
& \ \ \ \ \ \ \sqcup \left( \func{spec}\Delta _{r-1}^{(a)}(G^{n-1})+\func{%
spec}\Delta _{1}^{(a)}(G)\right) \ \ \ \ \ \ \ \ \ (p\underset{}{=}r-1,q%
\underset{}{=}1) \\
& =\{(\alpha _{1}+\cdots +\alpha _{n-r-1}+\beta _{1}+\cdots +\beta _{r})_{%
\binom{n-1}{r}}+\alpha _{n-r}\} \\
& \quad \ \ \sqcup \{(\alpha _{1}+\cdots +\alpha _{(n-1)-(r-1)}+\beta
_{1}+\cdots +\beta _{r-1})_{\binom{n-1}{r-1}}+\beta _{r}\} \\
& =\{(\alpha _{1}+\cdots +\alpha _{n-r}+\beta _{1}+\cdots +\beta _{r})_{%
\binom{n}{r}}\},
\end{align*}%
where $\alpha _{i}\in \func{spec}\Delta _{0}^{(a)}(G)$ and $\beta _{j}\in
\func{spec}\Delta _{1}^{(a)}(G)$.
\end{proof}

Here is our main result about the spectrum of the normalized Hodge Laplacian.

\begin{theorem}
\label{TspecDeaGn}Assume that $G$ is a connected digraph containing neither
double arrow, nor triangle nor square. Let $\lambda _{1},...,\lambda _{s}$
be all distinct positive eigenvalues of $\Delta _{0}(G)$, $0<\lambda
_{1}<...<\lambda _{s}$, and let $m_{1},...,m_{s}$ be their multiplicities.

$\left( a\right) $ Assume that $G$ has $v$ vertices and $v$ arrows, where $%
v\geq 3.$ Then, for all $0\leq r\leq n$, we have%
\begin{eqnarray}
\func{spec}\Delta _{r}^{(a)}(G^{n}) &=&\left\{ \left( k_{1}\lambda
_{1}+...+k_{s}\lambda _{s}\right) _{m_{1}^{k_{1}}...m_{s}^{k_{s}}\binom{n}{%
k_{1}\,\ ...\ \ k_{s}\,}\binom{n}{r}}\ \ \right\vert  \notag \\
&&\ \ \ \ \ \ \ \left. k_{1},...,k_{s}\geq 0,\ k_{1}+...+k_{s}\underset{}{%
\leq }n\ \right\} ,  \label{kla}
\end{eqnarray}%
where $\binom{n}{k_{1}\,\ ...\ \ k_{s}\,}$ is a multinomial coefficient.
Consequently,
\begin{equation*}
\lambda _{\max }(\Delta _{r}^{(a)}(G^{n}))=\left( n\lambda _{s}\right)
_{m_{s}^{n}\binom{n}{r}}\ \ \ \text{and\ \ \ }\lambda _{\min }(\Delta
_{r}^{(a)}(G^{n}))=0_{\binom{n}{r}}.
\end{equation*}

$\left( b\right) \ $Assume that $G$ has $v$ vertices and $v-1$ arrows, where
$v\geq 2.$ Then, for all $0\leq r\leq n$, we have%
\begin{align}
\func{spec}\Delta _{r}^{(a)}(G^{n})& =\left\{ (k_{1}\lambda
_{1}+...+k_{s}\lambda _{s})_{m_{1}^{k_{1}}...m_{s}^{k_{s}}\binom{n}{k_{1}\
...\ k_{s}}\binom{k_{1}+...+k_{s}}{r}}\,~\right\vert  \notag \\
& \ \ \ \ \ \ \ \ \ \ \ \ \ \ \left. k_{1},...,k_{s}\geq 0,\ r\leq
k_{1}+...+k_{s}\underset{}{\leq }n\right\} .  \label{kla1}
\end{align}%
Consequently,
\begin{equation*}
\lambda _{\max }(\Delta _{r}^{(a)}(G^{n}))=\left( n\lambda _{s}\right)
_{m_{s}^{n}\binom{n}{r}}\ \ \ \text{and\ \ \ }\lambda _{\min }(\Delta
_{r}^{(a)}(G^{n}))=\left\{ r\lambda _{1}\right\} _{m_{1}^{r}\binom{n}{r}}.
\end{equation*}
\end{theorem}

\begin{proof}
By (\ref{De0n}) we have%
\begin{equation*}
\Delta _{0}^{\left( a\right) }(G)=\Delta _{0}(G)
\end{equation*}%
so that we can replace in (\ref{ab}) $\func{spec}\Delta _{0}^{\left(
a\right) }(G)$ by $\func{spec}\Delta _{0}(G)$.

Let us handle $\func{spec}\Delta _{1}^{\left( a\right) }(G)$ in (\ref{ab}).
By Proposition \ref{POm2} we have $\Omega _{p}(G)=\left\{ 0\right\} $ for
all $p\geq 2$. Since $\Omega _{2}(G)=\left\{ 0\right\} $ we have by Example %
\ref{ExOm2=0} that
\begin{equation}
\func{spec}_{+}\Delta _{1}^{\left( a\right) }(G)=\func{spec}_{+}\Delta
_{0}^{\left( a\right) }(G).  \label{1=0}
\end{equation}%
Since $G$ is connected, we have $\dim H_{0}(G)=1$, so that the eigenvalue $0$
of $\Delta _{0}^{\left( a\right) }(G)$ has the multiplicity $1$. The
multiplicity of $0$ as the eigenvalue of $\Delta _{1}^{\left( a\right) }(G)$
(and $\Delta _{1}(G)$) is equal to $\dim H_{1}(G).$ To determine it, we use
the Euler characteristic $\chi $ of the chain complex $\Omega _{\ast }$ that
in this case satisfies
\begin{equation*}
\chi =\dim \Omega _{0}-\dim \Omega _{1}=\dim H_{0}-\dim H_{1}.
\end{equation*}%
It follows that
\begin{equation}
\dim H_{1}(G)=1-\left\vert V\right\vert +\left\vert E\right\vert .
\label{dimH1}
\end{equation}

$\left( a\right) $ Since $\left\vert V\right\vert =\left\vert E\right\vert
=v $, we obtain from (\ref{dimH1}) $\dim H_{1}=1$. Hence, $0$ has the same
multiplicity $1$ in $\func{spec}\Delta _{1}^{\left( a\right) }(G)$ and $%
\func{spec}\Delta _{0}^{\left( a\right) }(G)$. Combining with (\ref{1=0}) we
obtain%
\begin{equation*}
\func{spec}\Delta _{1}^{\left( a\right) }(G)=\func{spec}\Delta _{0}^{\left(
a\right) }(G)=\func{spec}\Delta _{0}(G).
\end{equation*}%
Therefore, we can replace in (\ref{ab}) the both spectra $\func{spec}\Delta
_{0}^{\left( a\right) }(G)$ and $\func{spec}\Delta _{1}^{\left( a\right)
}(G) $ by $\func{spec}\Delta _{0}(G)$, so that
\begin{equation}
\func{spec}\Delta _{r}^{(a)}(G^{n})=\{(\alpha _{1}+\cdots +\alpha _{n})_{%
\binom{n}{r}}\mid \alpha _{i}\in \func{spec}\Delta _{0}(G)\}.  \label{aa}
\end{equation}%
Let $k_{l}$ be the number of times the eigenvalue $\lambda _{l}$ occurs in
the sequence $\alpha _{1},...,\alpha _{n}$. Then
\begin{equation*}
k_{1}+...+k_{s}\leq n
\end{equation*}%
and%
\begin{equation}
\alpha _{1}+...+\alpha _{n}=k_{1}\lambda _{1}+...+k_{s}\lambda _{s}.
\label{alk}
\end{equation}%
The numbers of ways of inserting $k_{l}$ values $\lambda _{l}$ in the
sequence $\alpha _{1},...,\alpha _{n}$ for all $l=1,...,n$ is equal to
\begin{equation*}
\binom{n}{k_{1}\ ...\ k_{s}}=\frac{n!}{k_{1}!...k_{s}!(n-k_{1}-...-k_{s})!}
\end{equation*}%
(and the rest $n-k_{1}-...-k_{s}$ values are $0$). Besides, within $k_{l}$
already fixed positions of $\lambda _{l}$, there are $m_{l}^{k_{l}}$ ways of
selecting this $\lambda _{l}$ from $m_{l}$ instances of $\lambda _{l}$ in $%
\func{spec}\Delta _{0}(G).$ Hence, the number of sequences $\alpha
_{1},...,\alpha _{n}$ where $\lambda _{l}$ occurs $k_{l}$ times, is equal to
\begin{equation*}
m_{1}^{k_{1}}...m_{s}^{k_{s}}\binom{n}{k_{1}\ ...\ k_{s}}.
\end{equation*}%
It follows from (\ref{aa}) that $k_{1}\lambda _{1}+...+k_{s}\lambda _{s}$
has the multiplicity
\begin{equation*}
m_{1}^{k_{1}}...m_{s}^{k_{s}}\binom{n}{k_{1}\ ...\ k_{s}}\binom{n}{r},
\end{equation*}%
which finishes the proof of (\ref{kla}).

$\left( b\right) $ Since $\left\vert V\right\vert =v$ and $\left\vert
E\right\vert =v-1$, we obtain from (\ref{dimH1}) $\dim H_{1}=0.$ Hence, $%
0\notin \Delta _{1}^{\left( a\right) }(G)$, and it follows from (\ref{1=0})
\begin{equation*}
\func{spec}\Delta _{1}^{(a)}(G)=\func{spec}_{+}\Delta _{0}^{\left( a\right)
}(G)=\func{spec}\Delta _{0}(G)\setminus \left\{ 0\right\} .
\end{equation*}%
Substituting into (\ref{ab}) we obtain
\begin{align}
\func{spec}\Delta _{r}^{(a)}(G^{n})& =\left\{ (\alpha _{1}+\cdots +\alpha
_{n-r}+\beta _{1}+\cdots +\beta _{r})_{\binom{n}{r}}\ \right\vert  \notag \\
& \ \ \ \ \ \ \ \ \ \left. \alpha _{i}\in \func{spec}\Delta _{0}(G),\,\beta
_{j}\underset{}{\in }\func{spec}\Delta _{0}(G)\setminus \left\{ 0\right\}
\right\} .  \label{ab0}
\end{align}%
Let $\lambda _{l}$ occur $k_{l}$ times in the sequence $\left\{ \alpha
_{i},\beta _{j}\right\} $ so that
\begin{equation*}
\alpha _{1}+\cdots +\alpha _{n-r}+\beta _{1}+\cdots +\beta _{r}=k_{1}\lambda
_{1}+...+k_{s}\lambda _{s}.
\end{equation*}%
Then there are $n-k_{1}-...-k_{s}$ values $0$ in the sequence $\left\{
\alpha _{i},\beta _{j}\right\} $, and they may be chosen only within $%
\left\{ \alpha _{i}\right\} $; hence, there are
\begin{equation*}
\binom{n-r}{n-k_{1}-...-k_{s}}=\binom{n-r}{k_{1}+...+k_{s}-r}
\end{equation*}%
ways of choosing positions of $0$ in $\left\{ \alpha _{i},\beta _{j}\right\}
.$ In the remaining $k_{1}+...+k_{s}$ positions we can place $k_{l}$ times $%
\lambda _{l}$ in
\begin{equation*}
\binom{k_{1}+...+k_{s}}{k_{1}\ \ ...\ \ k_{s}}
\end{equation*}%
ways. Taking into account the multiplicities $m_{l}$ as in $\left( a\right) $%
, we obtain that number of sequences $\alpha _{1},...,\alpha _{n}$ as above
is equal to%
\begin{equation*}
m_{1}^{k_{1}}...m_{s}^{k_{s}}\binom{n-r}{k_{1}+...+k_{s}-r}\binom{%
k_{1}+...+k_{s}}{k_{1}\ \ ...\ \ k_{s}},
\end{equation*}%
which gives the multiplicity of $k_{1}\lambda _{1}+...+k_{s}\lambda _{s}$ as%
\begin{equation*}
m_{1}^{k_{1}}...m_{s}^{k_{s}}\binom{n-r}{k_{1}+...+k_{s}-r}\binom{%
k_{1}+...+k_{s}}{k_{1}\ \ ...\ \ k_{s}}\binom{n}{r}.
\end{equation*}%
Setting $k=k_{1}+...+k_{s}$, let us verify that
\begin{equation*}
\binom{n-r}{k-r}\binom{k}{k_{1}\ \ ...\ \ k_{s}}\binom{n}{r}=\binom{n}{%
k_{1}\ ...\ k_{s}}\binom{k}{r},
\end{equation*}%
which will conclude the proof of (\ref{kla1}). Indeed, we have%
\begin{eqnarray*}
\binom{n-r}{k-r}\binom{k}{k_{1}\ \ ...\ \ k_{s}}\binom{n}{r} &=&\frac{\left(
n-r\right) !}{\left( n-k\right) !\left( k-r\right) !}\frac{k!}{%
k_{1}!...k_{s}!}\frac{n!}{(n-r)!r!} \\
&=&\frac{n!}{\left( n-k\right) !k_{1}!...k_{s}!}\frac{k!}{(k-r)!r!}=\binom{n%
}{k_{1}\ ...\ k_{s}}\binom{k}{r},
\end{eqnarray*}%
which was claimed.
\end{proof}

In the next examples we compute the spectra $\func{spec}\Delta
_{r}^{(a)}(G^{n})$ for six small digraphs $G$ shown n Fig. \ref{pic6}

\FRAME{ftbphFU}{5.7449in}{0.9357in}{0pt}{\Qcb{Some examples of digraphs}}{%
\Qlb{pic6}}{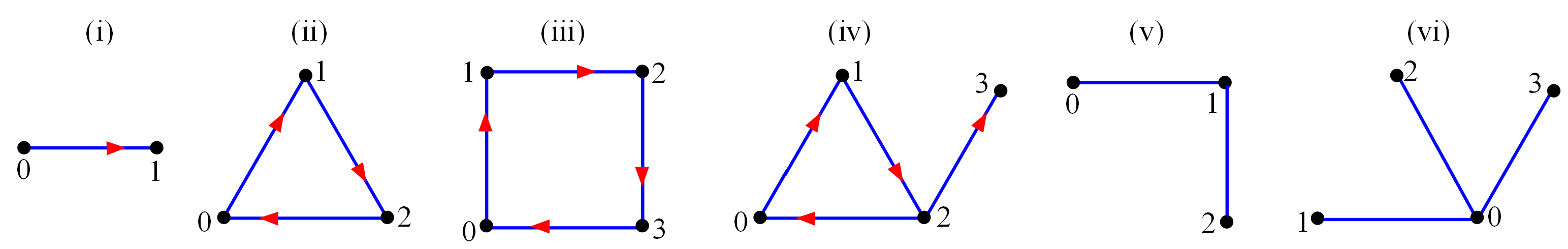}{\special{language "Scientific Word";type
"GRAPHIC";maintain-aspect-ratio TRUE;display "USEDEF";valid_file "F";width
5.7449in;height 0.9357in;depth 0pt;original-width 15.9445in;original-height
2.5694in;cropleft "0";croptop "1";cropright "1";cropbottom "0";filename
'pic6.png';file-properties "XNPEU";}}

\begin{example}
\label{ExDea}Case $\left( \mathrm{i}\right) $. For the interval $G=I=\left\{
0\rightarrow 1\right\} $ the hypotheses of $\left( b\right) $ are satisfied
with $v=2$. The matrix of $\Delta _{0}(I)$ is
\begin{equation*}
\begin{pmatrix}
1 & -1 \\
-1 & 1%
\end{pmatrix}%
\end{equation*}%
and $\func{spec}\Delta _{0}(I)=\left\{ 0,2\right\} .$ Hence, we have only
one positive eigenvalue $\lambda _{1}=2$ with multiplicity $m_{1}=1$. It
follows from (\ref{kla1}) that
\begin{equation}
\func{spec}\Delta _{r}^{(a)}(I^{n})=\left\{ (2k)_{\binom{n}{k}\binom{k}{r}%
}\right\} _{k=r}^{n}.  \label{specDeacube}
\end{equation}%
In particular, have
\begin{equation*}
\lambda _{\max }(\Delta _{r}^{(a)}(I^{n}))=\left( 2n\right) _{\binom{n}{r}}\
\text{and\ \ \ }\lambda _{\min }(\Delta _{r}^{(a)}(I^{n}))=\left( 2r\right)
_{\binom{n}{r}}.
\end{equation*}
\end{example}

\begin{example}
\label{ExDeaT}Case $\left( \mathrm{ii}\right) $. For the torus $G=T=\left\{
0\rightarrow 1\rightarrow 2\rightarrow 0\right\} $ the hypotheses of $\left(
a\right) $ are satisfied with $v=3$. The matrix of $\Delta _{0}(T)$ is%
\begin{equation*}
\begin{pmatrix}
2 & -1 & -1 \\
-1 & 2 & -1 \\
-1 & -1 & 2%
\end{pmatrix}%
\end{equation*}%
whence $\func{spec}\Delta _{0}(T)=\left\{ 0,3,3\right\} .$ Hence, we have
only one positive eigenvalue $\lambda _{1}=3$ with multiplicity $m_{1}=2.$
It follows from (\ref{kla}) that%
\begin{equation}
\func{spec}\Delta _{r}^{\left( a\right) }(T^{n})=\left\{ \left( 3k\right)
_{2^{k}\binom{n}{k}\binom{n}{r}}\right\} _{k=0}^{n}.  \label{specDeatorus}
\end{equation}%
Clearly, we have
\begin{equation*}
\lambda _{\max }(\Delta _{r}^{\left( a\right) }(T^{n}))=\left( 3n\right)
_{2^{n}\binom{n}{r}}\ \ \text{and\ \ \ }\lambda _{\min }(\Delta _{r}^{\left(
a\right) }(T^{n}))=0_{\binom{n}{r}}.
\end{equation*}
\end{example}

\begin{example}
\label{Excycle4}Case $\left( \mathrm{iii}\right) $. Let $G=\left\{
0\rightarrow 1\rightarrow 2\rightarrow 3\rightarrow 0\right\} $ so that the
hypotheses of $\left( a\right) $ are satisfied with $v=4$. Then the matrix
of $\Delta _{0}(G)$ is%
\begin{equation}
\begin{pmatrix}
2 & -1 & 0 & -1 \\
-1 & 2 & -1 & 0 \\
0 & -1 & 2 & -1 \\
-1 & 0 & -1 & 2%
\end{pmatrix}%
,  \label{cycle4}
\end{equation}%
whence $\func{spec}\Delta _{0}(G)=\left\{ 0,2,2,4\right\} .$Setting $\lambda
_{1}=2$,$\ m_{1}=2\ $and $\lambda _{2}=4$, $m_{2}=1$, we obtain by (\ref{kla}%
)
\begin{equation*}
\func{spec}\Delta _{r}^{(a)}(G^{n})=\left\{ \left( 2k_{1}+4k_{2}\right)
_{2^{k_{1}}\binom{n}{k_{1}\ k_{2}}\binom{n}{r}}\right\} _{0\leq
k_{1}+k_{2}\leq n}.
\end{equation*}%
It follows that
\begin{equation*}
\lambda _{\max }(\Delta _{r}^{(a)}(G^{n}))=\left( 4n\right) _{\binom{n}{r}}\
\ \text{and\ \ \ }\lambda _{\min }(\Delta _{r}^{(a)}(G^{n}))=0_{\binom{n}{r}%
}.
\end{equation*}
\end{example}

\begin{example}
Case $\left( \mathrm{iv}\right) $. Let $G=\left\{ 0\rightarrow 1\rightarrow
2\rightarrow 0\rightarrow 3\right\} $ so that $G$ satisfies the hypothesis $%
\left( a\right) $ with $v=4.$ The matrix of $\Delta _{0}(G)$ is
\begin{equation*}
\begin{pmatrix}
3 & -1 & -1 & -1 \\
-1 & 2 & -1 & 0 \\
-1 & -1 & 2 & 0 \\
-1 & 0 & 0 & 1%
\end{pmatrix}%
,
\end{equation*}%
the eigenvalues are $\left\{ 0,1,3,4\right\} $. We obtain by (\ref{kla})%
\begin{equation*}
\func{spec}\Delta _{r}^{(a)}(G^{n})=\left\{ \left(
k_{1}+3k_{2}+4k_{3}\right) _{\binom{n}{k_{1}\,k_{2}\ k_{3}\,}\binom{n}{r}}\
\ \right\} _{0\leq k_{1}+k_{2}+k_{2}\leq n}.
\end{equation*}
\end{example}

\begin{example}
Case $\left( \mathrm{v}\right) $. Let $G=\left\{ 0\sim 1\sim 2\right\} $
where the orientation of edges is arbitrary. Then the hypotheses of $\left(
b\right) $ are satisfied with $v=3$. The matrix of $\Delta _{0}(G)$ is
\begin{equation*}
\begin{pmatrix}
1 & -1 & 0 \\
-1 & 2 & -1 \\
0 & -1 & 1%
\end{pmatrix}%
,
\end{equation*}%
and the eigenvalues are $\left\{ 0,1,3\right\} .$ Setting in (\ref{kla1}) $%
\lambda _{1}=1$, $\lambda _{2}=3$ and $m_{1}=m_{2}=1$, we obtain%
\begin{equation*}
\func{spec}\Delta _{r}^{(a)}(G^{n})=\left\{ (k_{1}+3k_{2})_{\binom{n}{k_{1}\
\ k_{2}}\binom{k_{1}+k_{2}}{r}}\right\} _{r\leq k_{1}+k_{2}\leq n}.
\end{equation*}%
In particular,
\begin{equation*}
\lambda _{\max }(\Delta _{r}^{(a)}(G^{n}))=\left( 3n\right) _{\binom{n}{r}}\
\ \text{and ~}\lambda _{\min }(\Delta _{r}^{(a)}(G^{n}))=r_{\binom{n}{r}}.
\end{equation*}
\end{example}

\begin{example}
Case $\left( \mathrm{vi}\right) $. Let $G=\left\{ 0\sim 1,0\sim 2,0\sim
3\right\} $ where the orientation of edges is arbitrary. Then the hypotheses
of $\left( b\right) $ are satisfied with $v=4$. The matrix of $\Delta
_{0}(G) $ is
\begin{equation*}
\begin{pmatrix}
3 & -1 & -1 & -1 \\
-1 & 1 & 0 & 0 \\
-1 & 0 & 1 & 0 \\
-1 & 0 & 0 & 1%
\end{pmatrix}%
,
\end{equation*}%
$\allowbreak $and the eigenvalues of $\Delta _{0}(G)$ are $\left\{
0,1,1,4\right\} $. Setting in (\ref{kla1}) $\lambda _{1}=1$, $\lambda _{2}=4$%
, $m_{1}=2$ and $m_{2}=1$, we obtain
\begin{equation*}
\func{spec}\Delta _{r}^{(a)}(G^{n})=\left\{ (k_{1}+4k_{2})_{2^{k_{1}}\binom{n%
}{k_{1}\ \ k_{2}}\binom{k_{1}+.k_{2}}{r}}\right\} _{r\leq k_{1}+k_{2}\leq n}.
\end{equation*}%
In particular,
\begin{equation*}
\lambda _{\max }(\Delta _{r}^{(a)}(G^{n}))=\left( 4n\right) _{\binom{n}{r}}\
\ \ \text{and\ \ \ }\lambda _{\min }(\Delta _{r}^{(a)}(G^{n}))=\left\{
r\right\} _{2^{r}\binom{n}{r}}.
\end{equation*}
\end{example}

\subsection{Spectrum of the canonical Hodge Laplacian on $G^{n}$}

Let us define the \emph{Hodge spectrum} of a digraph $G$ as a sequence
\begin{equation*}
\func{spec}G:=\left\{ \func{spec}\Delta _{p}(G)\right\} _{p=0}^{\infty },
\end{equation*}%
where $\Delta _{p}$ is the canonical Hodge Laplacian, as well as the \emph{%
normalized Hodge spectrum }by
\begin{equation*}
\func{spec}^{\left( a\right) }G:=\left\{ \func{spec}\Delta _{p}^{\left(
a\right) }(G)\right\} _{p=0}^{\infty },
\end{equation*}%
where $a$ is the weight (\ref{ap}) and, hence, $\Delta _{p}^{(a)}$ is the
normalized Hodge Laplacian.

\begin{theorem}
\label{thm:specGn}For any finite digraph $G$ and any $n\geq 1$, the
following is true.

$\left( a\right) $ $\func{spec}G^{n}$\ is determined by $\func{spec}^{(a)}G$.

$\left( b\right) $ $\func{spec}G^{n}$\ is determined by $\func{spec}G$.

$\left( c\right) $ If $\Omega _{r}(G)=\left\{ 0\right\} $ for some $r\geq 2$
then $\func{spec}G^{n}$\ is determined by the Euler characteristic $\chi (G)
$ and the sequence
\begin{equation}
\left\{ \func{spec}\Delta _{q}(G)\right\} _{q=0}^{r-2}.\   \label{q-2}
\end{equation}
\end{theorem}

\begin{proof}
$\left( a\right) $ Applying Lemma~\ref{lem:specDp''} to operators $\Delta
_{p}^{(a)}$ on $G^{n}$, we obtain that $\func{spec}_{+}\mathcal{L}%
_{p}^{(a)}(G^{n})$ is determined by the sequence
\begin{equation*}
\left\{ \func{spec}_{+}\Delta _{q}^{(a)}(G^{n})\right\} _{q=0}^{p-1}
\end{equation*}
and, hence, by $\func{spec}^{(a)}G^{n}.$ Since by (\ref{D''})
\begin{equation*}
\mathcal{L}_{p}^{(a)}=p\mathcal{L}_{p}\text{,}
\end{equation*}%
we see that $\func{spec}_{+}\mathcal{L}_{p}(G^{n})$ is also determined by $%
\func{spec}^{(a)}G^{n}$.

On the other hand, as it follows from Proposition~\ref{Tdeltaa}, $\func{spec}%
^{\left( a\right) }G^{n}$ is determined by $\func{spec}^{\left( a\right) }G.$
Hence, $\func{spec}_{+}\mathcal{L}_{p}(G^{n})$ is determined by $\func{spec}%
^{\left( a\right) }G.$

By (\ref{spec+}) we have
\begin{equation*}
\func{spec}_{+}\Delta _{p}(G^{n})=\func{spec}_{+}\mathcal{L}%
_{p}(G^{n})\sqcup \func{spec}_{+}\mathcal{L}_{p+1}(G^{n}),
\end{equation*}%
which implies that $\func{spec}_{+}\Delta _{p}(G^{n})$ is also determined by
$\func{spec}^{\left( a\right) }G$.

Finally, the multiplicity of $0$ as an eigenvalue of $\Delta _{p}(G^{n})$ is
equal to $\dim H_{p}(G^{n})$, which by the K\"{u}nneth formula (\ref{kunH})
is determined by $\left\{ \dim H_{p}(G)\right\} _{p\geq 0}$ and, hence, by $%
\func{spec}^{\left( a\right) }G.$ Hence, the spectrum $\func{spec}\Delta
_{p}(G^{n})$ is determined by $\func{spec}^{\left( a\right) }G$, which
finishes the proof.

$(b)$ This follows from $\left( a\right) $ and Proposition \ref{Pspeca}.

$(c)$ By Corollary \ref{Corspeca}, for any $p<r$, $\func{spec}\Delta
_{p}^{(a)}(G)$ is determined by $\chi (G)$ and the sequence $\left\{ \func{%
spec}\Delta _{q}(G)\right\} _{q=0}^{p-1}$ and, hence, by $\left\{ \func{spec}%
\Delta _{q}(G)\right\} _{q=0}^{r-2}.$ Hence, the rest follows from $(a)$.
\end{proof}

The argument in the proof of Theorem \ref{thm:specGn} theoretically allows
to compute $\func{spec}G^{n}$ knowing $\func{spec}^{\left( a\right) }G$.
However, practically it is more convenient to compute first $\func{spec}%
\Delta _{p}^{(a)}(G^{n})$ by means of Proposition \ref{Tdeltaa} or Theorem %
\ref{TspecDeaGn}, and then compute $\func{spec}\mathcal{L}_{p}(G^{n})$ and $%
\func{spec}\Delta _{p}(G^{n})$ using the above argument. In the next section
we apply this approach in order to compute the Hodge spectrum for $n$-cube
and $n$-torus.

\subsection{Isospectral digraphs}

\label{SecIso}We say that two digraphs $G$ and $G^{\prime }$are \emph{Hodge
isospectral }if $\func{spec}G=\func{spec}G^{\prime }$. A natural question in
the spirit of inverse spectral problems is whether Hodge isospectral
digraphs are isomorphic. We show here that in general the answer is
\textquotedblleft no\textquotedblright .

Let $\overline{G}=(V,E)$ be a connected undirected graph such that $%
\overline{G}$ contains neither $3$-cycles nor $4$-cycles. For example, $%
\overline{G}$ can be a polygon with $n\geq 5$ sides.

Let $G$ be a digraph that is obtained from $\overline{G}$ by assigning
arbitrarily orientation on each edge thus turning it into an arrow. Then $%
\Omega _{2}(G)=\left\{ 0\right\} $, and, by Theorem \ref{thm:specGn}$\left(
c\right) $ with $r=2$, we obtain that $\func{spec}G^{n}$ is determined by $%
\func{spec}\Delta _{0}$ and the Euler characteristic $\chi =\left\vert
V\right\vert -\left\vert E\right\vert $. Since both $\Delta _{0}$ and $\chi $
are independent of orientation of the edges (cf. (\ref{made0})\textbf{)}, we
conclude that also $\func{spec}G^{n}$ is independent of orientation of the
edges.

Hence, for any two digraphs $G_{1}$ and $G_{2}$ that are obtained by
assigning orientation of edges in $\overline{G}$, the digraphs $G_{1}^{n}$
and $G_{2}^{n}$ are Hodge isospectral for any $n\geq 1$, but obviously they
do not have to be isomorphic.

However, $G_{1}^{n}$ and $G_{2}^{n}$ are still isomorphic as undirected
graphs.

\begin{problem}
Is it true that Hodge isospectral digraphs are isomorphic as undirected
graphs?
\end{problem}

\section{Spectrum of the Hodge Laplacian on cubes and tori}

\label{sec:cube}\setcounter{equation}{0}Recall that the $n$-cube $I^{n}$ and
$n$-torus $T^{n}$ were defined in Examples \ref{ExCube} and \ref{ExTorus},
respectively. In this section we compute the Hodge spectra of $I^{n}$ and $%
T^{n}.$ Since for both digraphs $G=I$ and $G=T$ we have $\Omega
_{2}(G)=\left\{ 0\right\} $, the space $\Omega _{p}(G^{n})$ is trivial for $%
p>n$ and, hence, it suffices to compute $\func{spec}\Delta _{p}(G^{n})$ for $%
p\leq n$. For $p=0$ we obtain by (\ref{specDeacube})
\begin{equation}
\func{spec}\Delta _{0}(I^{n})=\func{spec}\Delta _{0}^{\left( a\right)
}(I^{n})=\left\{ \left( 2k\right) _{\binom{n}{k}}\right\} _{k=0}^{n}
\label{De0In}
\end{equation}%
and by (\ref{specDeatorus})%
\begin{equation}
\func{spec}\Delta _{0}(T^{n})=\func{spec}\Delta _{0}^{\left( a\right)
}(T^{n})=\left\{ \left( 3k\right) _{2^{k}\binom{n}{k}}\right\} _{k=0}^{n}.
\label{De0Tn}
\end{equation}%
Hence, in what follows we restrict ourselves to $1\leq p\leq n$.

Now we state and prove the main result about $\func{spec}I^{n}.$

\begin{theorem}
\label{TspecDepIn}For all $1\leq p\leq n$ we have%
\begin{equation}
\func{spec}\Delta _{p}(I^{n})=\left\{ \left( \frac{2k}{p}\right) _{\binom{n}{%
k}\binom{k-1}{p-1}}\right\} _{k=p}^{n}\sqcup \left\{ \left( \frac{2k}{p+1}%
\right) _{\binom{n}{k}\binom{k-1}{p}}\right\} _{k=p+1}^{n}.
\label{specDepIn}
\end{equation}%
In particular,%
\begin{equation*}
\lambda _{\max }\left( \Delta _{p}(I^{n})\right) =\left( \frac{2n}{p}\right)
_{\binom{n-1}{p-1}}\ \ \text{and\ \ }\lambda _{\min }\left( \Delta
_{p}(I^{n})\right) =2_{\binom{n+1}{p+1}}
\end{equation*}
\end{theorem}

\begin{proof}
\label{LemLpIn}We start with computation of the spectrum of $\mathcal{L}%
_{p}(I^{n})$. Namely, let us first prove that, for $1\leq p\leq n,$
\begin{equation}
\func{spec}_{+}\mathcal{L}_{p}(I^{n})=\left\{ \left( \frac{2k}{p}\right) _{%
\binom{n}{k}\binom{k-1}{p-1}}\right\} _{k=p}^{n}.  \label{D''In}
\end{equation}%
It follows from (\ref{De0In}) that
\begin{equation*}
\func{spec}_{+}\mathcal{L}_{1}(I^{n})=\func{spec}_{+}\Delta
_{0}(I^{n})=\left\{ \left( 2k\right) _{\binom{n}{k}}\right\} _{k=1}^{n},
\end{equation*}%
which matches (\ref{D''In}) for $p=1$ and provides for the induction basis.

For the induction step from $p$ to $p+1$, observe that, by (\ref{specDen+}),
\begin{equation}
\func{spec}_{+}\Delta _{p}^{\left( a\right) }=p\func{spec}_{+}\mathcal{L}%
_{p}\,\sqcup \,(p+1)\func{spec}_{+}\mathcal{L}_{p+1},  \label{specDep}
\end{equation}%
whence%
\begin{equation}
\left( p+1\right) \func{spec}_{+}\mathcal{L}_{p+1}=\func{spec}_{+}\Delta
_{p}^{(a)}\setminus p\func{spec}_{+}\mathcal{L}_{p}.  \label{p+1}
\end{equation}%
By (\ref{specDeacube}) we have
\begin{equation}
\func{spec}\Delta _{p}^{\left( a\right) }(I^{n})=\left\{ \left( 2k\right) _{%
\binom{n}{k}\binom{k}{p}}\right\} _{k=p}^{n}.  \label{specDeacube1}
\end{equation}%
Using also the induction hypothesis, we obtain
\begin{align*}
\left( p+1\right) \func{spec}_{+}\mathcal{L}_{p+1}(I^{n})& =\left\{ \left(
2k\right) _{\binom{n}{k}\binom{k}{p}}\right\} _{k=p}^{n}\setminus p\left\{
\left( \frac{2k}{p}\right) _{\binom{n}{k}\binom{k-1}{p-1}}\right\} _{k=p}^{n}
\\
& =\left\{ \left( 2k\right) _{\binom{n}{k}\binom{k-1}{p}}\right\} _{k=p}^{n}.
\end{align*}%
It follows that%
\begin{equation*}
\func{spec}_{+}\mathcal{L}_{p+1}=\left\{ \left( \frac{2k}{p+1}\right) _{%
\binom{n}{k}\binom{k-1}{p}}\right\} _{k=p+1}^{n},
\end{equation*}%
(where now $k$ starts from $p+1$ because $\binom{k-1}{p}=0$ for $k=p$),
which concludes the induction step.

Now we can prove (\ref{specDepIn}). Since $I$ is homologically trivial, also
all cubes $G=I^{n}$ are homologically trivial, which implies that, for all $%
1\leq p\leq n,$
\begin{equation*}
\func{spec}\Delta _{p}(I^{n})>0.
\end{equation*}%
Hence, by (\ref{spec+}) we have
\begin{equation*}
\func{spec}\Delta _{p}(I^{n})=\func{spec}_{+}\mathcal{L}_{p}(I^{n})\sqcup
\func{spec}_{+}\mathcal{L}_{p+1}(I^{n}).
\end{equation*}%
Substituting $\func{spec}_{+}\mathcal{L}_{p}(I^{n})$ and $\func{spec}_{+}%
\mathcal{L}_{p+1}(I^{n})$ from (\ref{D''In}), we obtain (\ref{specDepIn}).

The maximal eigenvalue $\lambda _{\max }=\frac{2n}{p}$ comes from the first
of two series in (\ref{specDepIn}) for $k=n$, with the multiplicity $\binom{n%
}{n}\binom{n-1}{p-1}=\binom{n-1}{p-1}.$ The minimal eigenvalue $\lambda
_{\min }=2$ comes from the both series, with $k=p$ and $k=p+1$,
respectively, and its multiplicity is%
\begin{equation*}
\binom{n}{p}\binom{p-1}{p-1}+\binom{n}{p+1}\binom{p}{p}=\binom{n+1}{p+1}.
\end{equation*}
\end{proof}

For example, in the case $p=1$ we obtain
\begin{equation}
\func{spec}\Delta _{1}(I^{n})=\left\{ \left( 2k\right) _{\binom{n}{k}%
}\right\} _{k=1}^{n}\sqcup \left\{ k_{_{\left( k-1\right) \binom{n}{k}%
}}\right\} _{k=2}^{n},  \label{specDe1In}
\end{equation}%
\begin{equation*}
\lambda _{\min }\left( \Delta _{1}(I^{n})\right) =2_{\binom{n}{2}}\ \ \
\text{and\ \ \ }\lambda _{\max }\left( \Delta _{1}(I^{n})\right) =\left(
2n\right) _{1},
\end{equation*}%
which solves Problem~6.26 in \cite{G2022}.

\label{SecDepTn}Similarly, we can compute the Hodge spectrum of $n$-torus $%
T^{n}.$

\begin{theorem}
\label{TspecDepTn}For all $1\leq p\leq n$ we have
\begin{equation}
\func{spec}\Delta _{p}(T^{n})=\left\{ \left( \frac{3k}{p}\right) _{2^{k}%
\binom{n}{k}\binom{n-1}{p-1}}\right\} _{k=0}^{n}\sqcup \left\{ \left( \frac{%
3k}{p+1}\right) _{2^{k}\binom{n}{k}\binom{n-1}{p}}\right\} _{k=0}^{n}.
\label{specDepTn}
\end{equation}%
In particular,
\begin{equation*}
\lambda _{\max }\left( \Delta _{p}(T^{n})\right) =\left( \frac{3n}{p}\right)
_{2^{n}\binom{n-1}{p-1}}\ \ \text{and}\ \ \ \lambda _{\min }\left( \Delta
_{p}(T^{n})\right) =0_{\binom{n}{p}}.
\end{equation*}
\end{theorem}

\begin{proof}
Let us first prove that, for all $1\leq p\leq n$,
\begin{equation}
\func{spec}_{+}\mathcal{L}_{p}(T^{n})=\left\{ \left( \frac{3k}{p}\right)
_{2^{k}\binom{n}{k}\binom{n-1}{p-1}}\right\} _{k=1}^{n}.  \label{D''Tn}
\end{equation}%
The induction basis for $p=1$ is given by (\ref{De0Tn}):
\begin{equation*}
\func{spec}_{+}\mathcal{L}_{1}(T^{n})=\func{spec}_{+}\Delta
_{0}(T^{n})=\left\{ \left( 3k\right) _{2^{k}\binom{n}{k}}\right\} _{k=1}^{n}.
\end{equation*}%
For the induction step from $p$ to $p+1$, we use (\ref{p+1}), the induction
hypothesis, and (\ref{specDeatorus}) in the form%
\begin{equation*}
\func{spec}_{+}\Delta _{p}^{\left( a\right) }(T^{n})=\left\{ \left(
3k\right) _{2^{k}\binom{n}{k}\binom{n}{p}}\right\} _{k=1}^{n},
\end{equation*}%
which yield
\begin{align*}
\left( p+1\right) \func{spec}_{+}\mathcal{L}_{p+1}(T^{n})& =\func{spec}%
_{+}\Delta _{p}^{(a)}(T^{n})\setminus p\func{spec}_{+}\mathcal{L}_{p}(T^{n})
\\
& =\left\{ \left( 3k\right) _{2^{k}\binom{n}{k}\binom{n}{p}}\right\}
_{k=1}^{n}\setminus p\left\{ \left( \frac{3k}{p}\right) _{2^{k}\binom{n}{k}%
\binom{n-1}{p-1}}\right\} _{k=1}^{n} \\
& =\left\{ \left( 3k\right) _{2^{k}\binom{n}{k}\binom{n-1}{p}}\right\}
_{k=1}^{n},
\end{align*}%
thus concluding the proof of (\ref{D''Tn}).

Now we can prove (\ref{specDepTn}). Indeed, by (\ref{spec+}) and (\ref{D''Tn}%
) we obtain
\begin{eqnarray*}
\func{spec}_{+}\Delta _{p}(T^{n}) &=&\func{spec}_{+}\mathcal{L}%
_{p}(T^{n})\sqcup \func{spec}_{+}\mathcal{L}_{p+1}(T^{n}) \\
&=&\left\{ \left( \frac{3k}{p}\right) _{2^{k}\binom{n}{k}\binom{n-1}{p-1}%
}\right\} _{k=1}^{n}\sqcup \left\{ \left( \frac{3k}{p+1}\right) _{2^{k}%
\binom{n}{k}\binom{n-1}{p}}\right\} _{k=1}^{n},
\end{eqnarray*}%
which matches the positive part of $\func{spec}\Delta _{p}(T^{n})$ in (\ref%
{SecDepTn}).

It remains to verify that $0$ is the eigenvalue of $\Delta _{p}(T^{n})$ with
multiplicity%
\begin{equation*}
\left. 2^{k}\binom{n}{k}\binom{n-1}{p-1}+2^{k}\binom{n}{k}\binom{n-1}{p}%
\right\vert _{k=0}=\binom{n}{p},
\end{equation*}%
which amounts to
\begin{equation}
\dim H_{p}(T^{n})=\binom{n}{p}.  \label{HpTn}
\end{equation}%
Indeed, since
\begin{equation*}
\dim H_{0}(T)=\dim H_{1}(T)=1,
\end{equation*}%
(\ref{HpTn}) follows from (\ref{kunH}) by induction in $n$.
\end{proof}

For example, for $p=1$ we obtain
\begin{equation*}
\func{spec}\Delta _{1}(T^{n})=\left\{ \left( 3k\right) _{2^{k}\binom{n}{k}%
}\right\} _{k=0}^{n}\sqcup \left\{ \left( \frac{3k}{2}\right) _{2^{k}\binom{n%
}{k}\left( n-1\right) }\right\} _{k=0}^{n}.
\end{equation*}

\section{Spectrum of the Hodge Laplacian on joins}

\label{sec:Dmn}\setcounter{equation}{0}In this section we use the augmented
chain complex on a digraph $G$:
\begin{equation}
\begin{array}{cccccccccccccc}
0 & \leftarrow & \mathbb{R} & \overset{\partial }{\leftarrow } & \Omega _{0}
& \overset{\partial }{\leftarrow } & \Omega _{1} & \overset{\partial }{%
\leftarrow } & \dots & \overset{\partial }{\leftarrow } & \Omega _{p-1} &
\overset{\partial }{\leftarrow } & \Omega _{p} & \overset{\partial }{%
\leftarrow }\dots%
\end{array}
\label{aug1}
\end{equation}%
where the boundary operator $\partial :\Omega _{0}\rightarrow \Omega _{-1}:=%
\mathbb{R}$ is now redefined\footnote{%
Recall that in Section \ref{sec:review} we defined $\partial e_{i}=0.$} by $%
\partial e_{i}=e$ where $e$ is the unity of $\mathbb{R}$.

In this section we use the canonical inner product $\left\langle \cdot
,\cdot \right\rangle $ on each $\Omega _{p}$ as in (\ref{canon}). Denote by $%
\widetilde{\Delta }_{p}$ the Hodge Laplacian associated with the chain
complex (\ref{aug1}). Of course, $\widetilde{\Delta }_{p}$ coincides with $%
\Delta _{p}$ for $p\geq 1$ but is different for $p=-1$ and $p=0.$ The
advantage of using the chain complex (\ref{aug1}) is that the operator $%
\widetilde{\Delta }_{p}$ satisfies the product rule with respect to the
operation \emph{join} of paths.

Let us briefly recall this notion based on \cite{GLMY2013}, \cite{GMY2017},
\cite{G2022}, \cite{GLY2024}.

For any two digraphs $X$ and $Y$, define their join as the digraph $Z=X\ast
Y $ whose set of vertices is a disjoint union of the set of vertices of $X$
and $Y$, and the set of arrows consists of all arrows in $X$ and $Y$ as well
as of all arrows $x\rightarrow y$ where $x\in X$ and $y\in Y$.

For any elementary paths $u=e_{i_{0}...i_{p}}$ on $X$ and $%
v=e_{j_{0}...j_{q}}$ on $Y$, define their join $u\ast v$ as a path on $Z$ by%
\begin{equation*}
u\ast v=e_{i_{0}...i_{p}j_{0}...j_{q}}.
\end{equation*}%
Observe that if $u$ and $v$ are allowed then $u\ast v$ is also allowed
because of the presence of the arrow $i_{p}\rightarrow j_{0}$. Note also
that the length of $u\ast v$ is $p+q+1$. Using linearity, this definition of
$u\ast v$ extends to all regular paths $u$ on $X$ and $v$ on $Y$.

The operator $\partial $ of the augmented chain complex (\ref{aug1})
satisfies the product rule: if $u\in \mathcal{R}_{p}(X)$ and $v\in \mathcal{R%
}_{q}(Y)$ with $p,q\geq -1$ then%
\begin{equation}
\partial (u\ast v)=\partial u\ast v+\left( -1\right) ^{p+1}u\ast \partial v.
\label{prj}
\end{equation}%
It implies that if $u\in \Omega _{p}(X)$ and $v\in \Omega _{q}(Y)$ then $%
u\ast v\in \Omega _{p+q+1}(Z).$

The augmented chain complex (\ref{aug1}) satisfies the K\"{u}nneth formula
with respect to join: for all $r\geq -1$%
\begin{equation}
\Omega _{r}\left( X\ast Y\right) \cong \tbigoplus_{\left\{ p,q\geq
-1:\,p+q+1=r\right\} }\left( \Omega _{p}\left( X\right) \otimes \Omega
_{q}\left( Y\right) \right) ,  \label{kunj}
\end{equation}%
where the isomorphism is given by $u\otimes v\mapsto u\ast v.$

It is easy to see that, $\ u\in \mathcal{A}_{p}\left( X\right) $, $v\in
\mathcal{A}_{q}\left( Y\right) $, $\varphi \in \mathcal{A}_{p^{\prime
}}\left( X\right) $ and $\psi \in \mathcal{A}_{q^{\prime }}\left( Y\right) $
then%
\begin{equation*}
\langle u\ast v,\varphi \ast \psi \rangle =\langle u,\varphi \rangle \langle
v,\psi \rangle .
\end{equation*}%
This together with (\ref{prj}) and (\ref{kunj}) allows to prove the product
rule for $\widetilde{\Delta }_{p}$ with respect to join, as in the following
statement.

\begin{proposition}
\label{LemDejoin}\emph{\cite[Lemma 5.5]{GLY2024}} Let $X,Y$ be two digraphs.
Then, for $u\in \Omega _{p}\left( X\right) $, $v\in \Omega _{q}\left(
Y\right) $ with $p.q\geq -1$ we have%
\begin{equation}
\widetilde{\Delta }_{r}\left( u\ast v\right) =(\widetilde{\Delta }_{p}u)\ast
v+u\ast \widetilde{\Delta }_{q}v,  \label{Deltajoin}
\end{equation}%
where $r=p+q+1$.
\end{proposition}

Combining (\ref{Deltajoin}) with (\ref{kunj}) as in the proof of Proposition %
\ref{Tdeltaa}, we obtain the following.

\begin{proposition}
\label{TXY}\emph{\cite[Theorem 6.36]{G2022}} We have for any $r\geq -1$
\begin{equation}
\func{spec}\widetilde{\Delta }_{r}\left( X\ast Y\right) =\tbigsqcup_{\left\{
p,q\geq -1:p+q=r-1\right\} }\left( \func{spec}\widetilde{\Delta }_{p}\left(
X\right) +\func{spec}\widetilde{\Delta }_{q}\left( Y\right) \right) .
\label{specjoin}
\end{equation}
\end{proposition}

Let $D_{m}$ denote the digraph that consists of $m\geq 1$ disjoint vertices
and no arrows, that is,
\begin{equation*}
D_{m}=\{\underset{n\text{ vertices}}{\underbrace{\bullet ,...,\bullet }}\}.
\end{equation*}%
Consider for any $n\geq 1$ the digraph
\begin{equation*}
D_{m}^{n}=\underset{n\text{\ times}}{\underbrace{D_{m}\ast ...\ast D_{m}}}.
\end{equation*}%
The next theorem is the main result of this section.

\begin{theorem}
\label{TspecDnm}We have, for all $n,m\geq 1$ and $r\geq 0$,%
\begin{equation}
\func{spec}\widetilde{\Delta }_{r-1}(D_{m}^{n})=\left\{ \left( (n-k)m\right)
_{(m-1)^{k}\binom{r}{k}\binom{n}{r}}\right\} _{k=0}^{r}.  \label{specDet}
\end{equation}%
Consequently, for all $r\geq 2$,%
\begin{equation}
\func{spec}\Delta _{r-1}(D_{m}^{n})=\left\{ \left( (n-k)m\right) _{(m-1)^{k}%
\binom{r}{k}\binom{n}{r}}\right\} _{k=0}^{r}.  \label{specDr-1}
\end{equation}
\end{theorem}

More explicitly, (\ref{specDet}) means the following: if $n<r$ then
\begin{equation*}
\func{spec}\widetilde{\Delta }_{r-1}(D_{m}^{n})=\emptyset ,
\end{equation*}%
while for $n\geq r$ the spectrum of $\widetilde{\Delta }_{r-1}(D_{m}^{n})$
consists of the following $r+1$ eigenvalues%
\begin{equation}
(n-r)m,\ (n-r+1)m,\ (n-r+2)m,...,(n-1)m,\ nm,  \label{eingenDnm}
\end{equation}%
with the multiplicities%
\begin{equation}
(m-1)^{r}\tbinom{n}{r},\ \ (m-1)^{r-1}r\tbinom{n}{r},\ (m-1)^{r-2}\tbinom{r}{%
2}\tbinom{n}{r},...,(m-1)r\tbinom{n}{r},\ \tbinom{n}{r}.  \label{mult}
\end{equation}

\begin{example}
\label{ExKn}Let $m=1$, that is, $D_{1}=\left\{ \bullet \right\} $. Clearly, $%
D_{1}^{n}$ coincides with a complete digraph $K_{n}$ that consist of $n$
vertices $\left\{ 0,...,n-1\right\} $ and all arrows $i<j$. Clearly, $K_{n}$
can be regarded an $\left( n-1\right) $-simplex digraph (see Fig. \ref{pic3}%
). \FRAME{ftbphFU}{3.9954in}{1.2393in}{0pt}{\Qcb{Simplices $K_{2}=D_{1}^{2}$
(interval), $K_{3}=D_{1}^{3}\ $(triangle) and $K_{4}=D_{2}^{4}$
(tetrahedron) }}{\Qlb{pic3}}{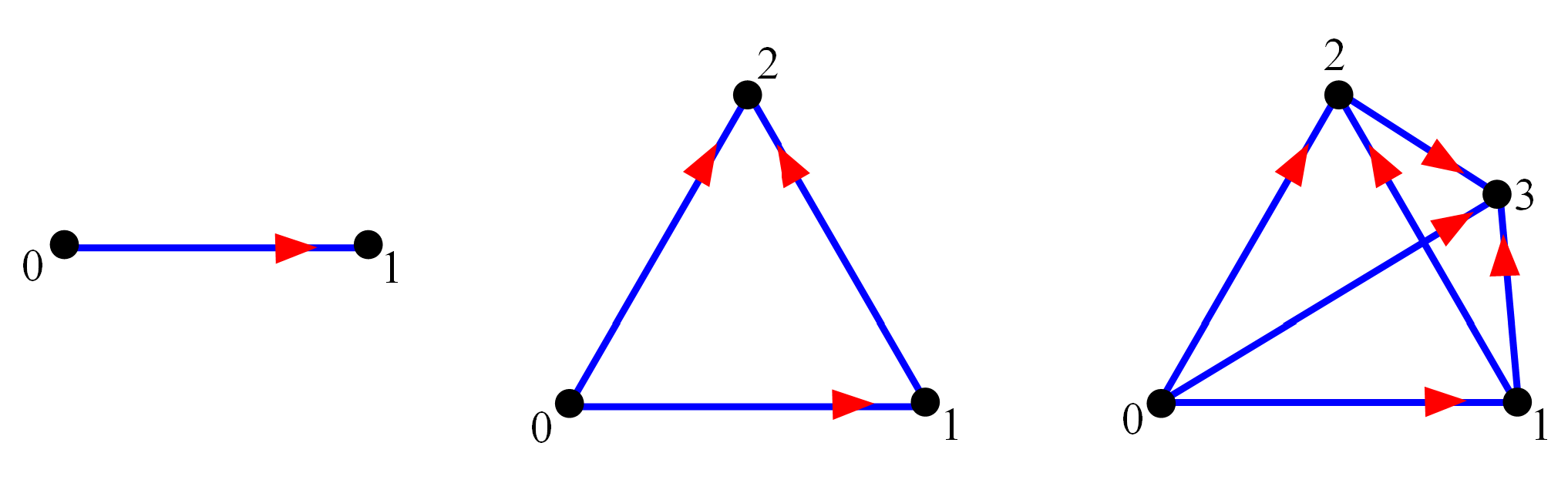}{\special{language "Scientific
Word";type "GRAPHIC";maintain-aspect-ratio TRUE;display "USEDEF";valid_file
"F";width 3.9954in;height 1.2393in;depth 0pt;original-width
9.4169in;original-height 2.9032in;cropleft "0";croptop "1";cropright
"1";cropbottom "0";filename 'pic3.png';file-properties "XNPEU";}}

In this case all the multiplicities in (\ref{mult}) are $0$ except for the
last one $\binom{n}{r}$. Hence, $\func{spec}\Delta _{r-1}(K_{n})$ consists
of a single eigenvalue $n$ with the multiplicity $\tbinom{n}{r}$.
\end{example}

\begin{example}
\RM Let $m=2$. $D_{2}=\left\{ \bullet ,\bullet \right\} $ and $%
D_{2}^{n}=:S^{n-1}$ can be regarded as a digraph sphere of dimension $n-1$.
For example, $S^{1}$ is a \emph{diamond} and $S^{2}$ is an \emph{octahedron}
as on Fig. \ref{pic266d}.\FRAME{ftbphFU}{3.9729in}{1.9043in}{0pt}{\Qcb{%
Digraph spheres: $S^{1}=D_{2}^{2}$ is a diamond, and $S^{2}=D_{2}^{3}$ is an
octahedron.}}{\Qlb{pic266d}}{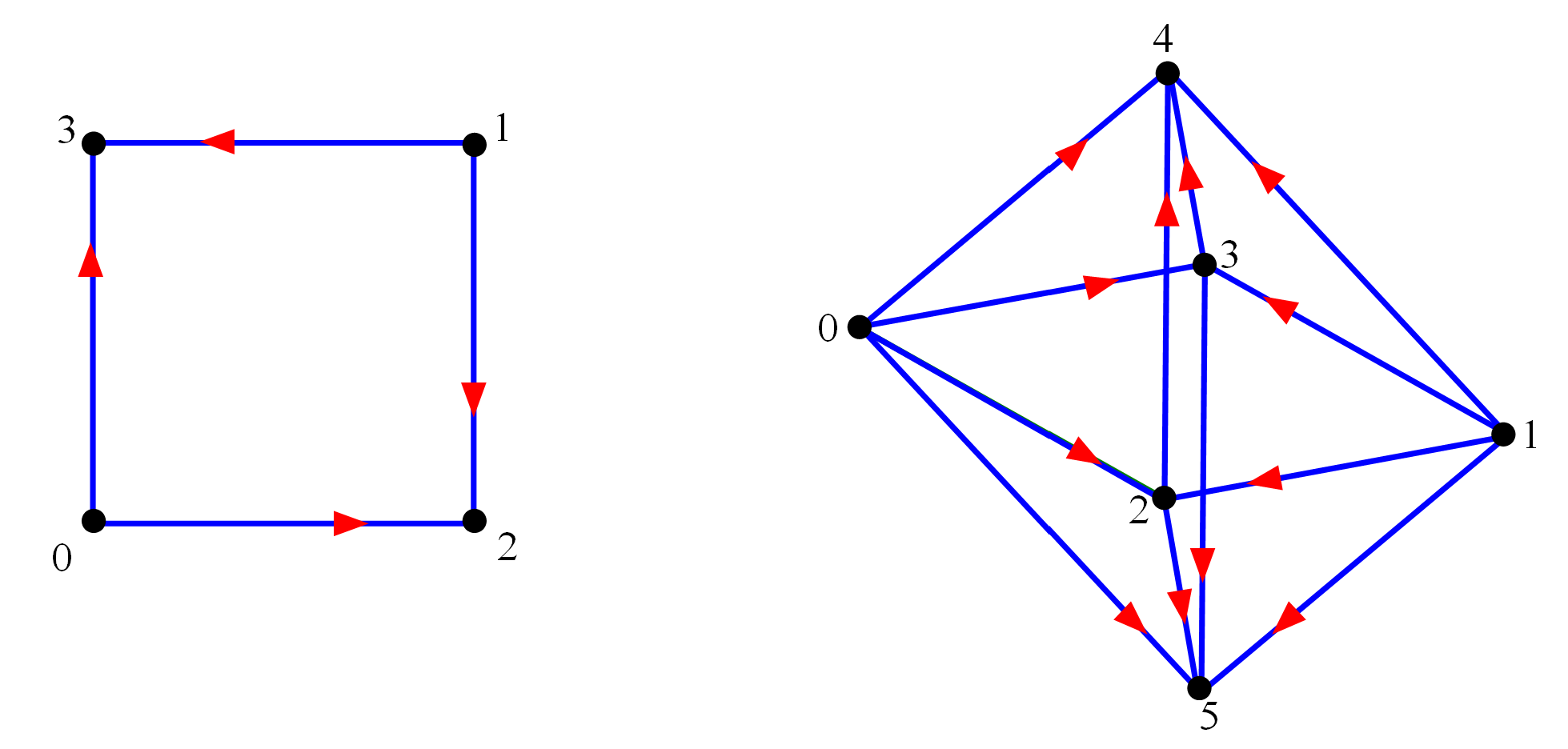}{\special{language "Scientific
Word";type "GRAPHIC";maintain-aspect-ratio TRUE;display "USEDEF";valid_file
"F";width 3.9729in;height 1.9043in;depth 0pt;original-width
9.0693in;original-height 4.3336in;cropleft "0";croptop "1";cropright
"1";cropbottom "0";filename 'pic266d.png';file-properties "XNPEU";}}

In this case (\ref{specDr-1}) becomes%
\begin{equation*}
\func{spec}\Delta _{r-1}(S^{n-1})=\left\{ \left( 2(n-k)\right) _{\binom{r}{k}%
\binom{n}{r}}\right\} _{k=0}^{r}.
\end{equation*}%
For example, for $r=2$ we have
\begin{equation*}
\func{spec}\Delta _{1}(S^{n-1})=\left\{ \left( 2(n-2)\right) _{\binom{n}{2}%
},\left( 2(n-1)\right) _{2\binom{n}{2}},\left( 2n\right) _{\binom{n}{2}%
}\right\} ,
\end{equation*}%
and for $r=3$%
\begin{equation*}
\func{spec}\Delta _{2}(S^{n-1})=\left\{ \left( 2(n-3)\right) _{\binom{n}{3}%
},\left( 2(n-2)\right) _{3\binom{n}{3}},\left( 2\left( n-1\right) \right) _{3%
\binom{n}{3}},\left( 2n\right) _{\binom{n}{3}}\right\} .
\end{equation*}
\end{example}

\begin{example}
\RM Let $m=3$ and $n=2.$ Then we have $D_{3}^{2}=K_{3,3}$ that is a complete
bipartite digraph as on Fig. \ref{pic33}. \FRAME{dtbphFU}{1.8152in}{1.2644in%
}{0pt}{\Qcb{Digraph $K_{3,3}$}}{\Qlb{pic33}}{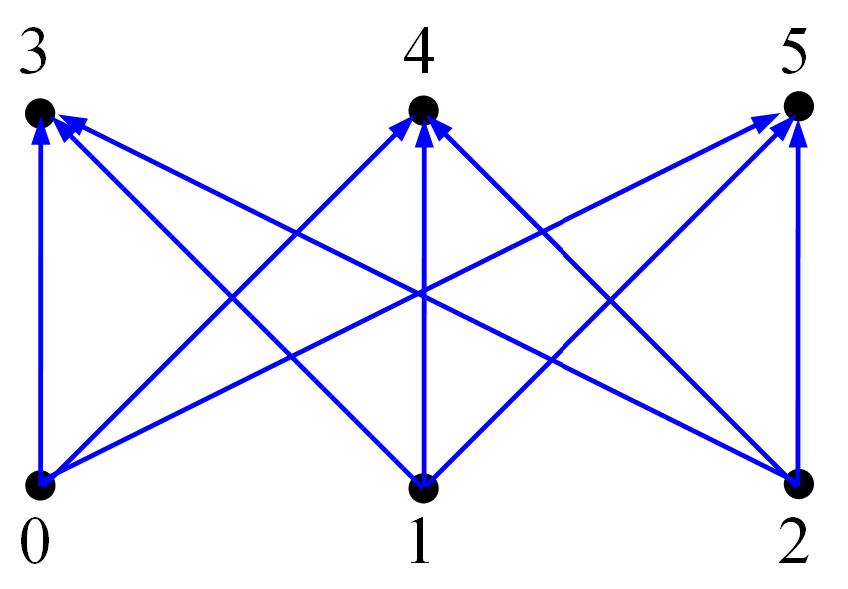}{\special{language
"Scientific Word";type "GRAPHIC";maintain-aspect-ratio TRUE;display
"USEDEF";valid_file "F";width 1.8152in;height 1.2644in;depth
0pt;original-width 3.9263in;original-height 2.7268in;cropleft "0";croptop
"1";cropright "1";cropbottom "0";filename 'pic33ar.png';file-properties
"XNPEU";}}

Then (\ref{specDr-1}) yields for $r=2$ that%
\begin{equation*}
\func{spec}\Delta _{1}(K_{3,3})=\left\{ \left( 3(2-k)\right) _{\binom{2}{k}%
\binom{2}{2}2^{k}}\right\} _{k=0}^{2}=\left\{ 0_{4},\,3_{4},\,6\right\} .
\end{equation*}
\end{example}

\begin{proof}[Proof of Theorem \protect\ref{TspecDnm}]
If $r\geq 2$ then $r-1\geq 1$ and $\Delta _{r-1}=\widetilde{\Delta }_{r-1}$
so that (\ref{specDr-1}) follows from (\ref{specDet}).

Let us first prove (\ref{specDet}) for $r=0$. We use the fact that the
operator $\widetilde{\Delta }_{-1}$ on any digraph is one-dimensional, and
it is multiplication by the number of vertices (see \cite[Section 6.9]{G2022}%
). Hence,
\begin{equation}
\func{spec}\widetilde{\Delta }_{-1}(D_{m}^{n})=\left\{ \left\vert
D_{m}^{n}\right\vert \right\} =\left\{ nm\right\} ,  \label{d-1}
\end{equation}%
which matches (\ref{specDet}) for $r=0.$

Let us prove (\ref{specDet}) for all $r\geq 1$ (and $m\geq 1$) by induction
in $n$.

Induction basis for $n=1$. If $r=1$ then the right hand side of (\ref%
{specDet}) is%
\begin{equation*}
\left\{ 0_{m-1},m\right\} .
\end{equation*}%
For any digraph $G$, the matrix of $\widetilde{\Delta }_{0}(G)$ in the basis
$\left\{ e_{i}\right\} $ is obtained from the matrix of $\Delta _{0}(G)$
(given by (\ref{made0})) by adding $1$ to all its entries (cf. \cite[Section
6.9]{G2022}). Since $\Delta _{0}(D_{m})=0$ by (\ref{made0}), it follows that
the matrix of $\widetilde{\Delta }_{0}(D_{m})$ is an $m\times m$ matrix with
all entries $=1$; hence%
\begin{equation}
\func{spec}\widetilde{\Delta }_{0}(D_{m})=\left\{ 0_{m-1},m\right\} ,
\label{d0}
\end{equation}%
which proves (\ref{specDet}) in this case. If $r\geq 2$ then
\begin{equation*}
\func{spec}\widetilde{\Delta }_{r-1}(D_{m})=\emptyset ,
\end{equation*}%
and the right hand side of (\ref{specDet}) is also empty as all the
multiplicities vanish. Hence, we have verified (\ref{specDet}) for $n=1.$

For the induction step from $n$ to $n+1$, let us use that
\begin{equation*}
D_{m}^{n+1}=D_{m}^{n}\ast D_{m},
\end{equation*}%
and%
\begin{equation*}
\left\vert D_{m}\right\vert =m\text{,\ \ }\left\vert D_{m}^{n}\right\vert
=nm.
\end{equation*}%
Let us apply (\ref{specjoin}) and rewrite it in the form%
\begin{equation*}
\func{spec}\widetilde{\Delta }_{r-1}(D_{m}^{n+1})=\bigsqcup_{\left\{ p,q\geq
0:p+q=r\right\} }\left( \func{spec}\widetilde{\Delta }_{p-1}\left(
D_{m}^{n}\right) +\func{spec}\widetilde{\Delta }_{q-1}\left( D_{m}\right)
\right) .
\end{equation*}%
The spectrum $\func{spec}\widetilde{\Delta }_{q-1}\left( D_{m}\right) $ is
empty if $q\geq 2$. Hence, the values of $q$ here should be restricted to $%
q=0$ and $q=1.$ Applying (\ref{d-1}) and (\ref{d0}) to compute $\func{spec}%
\widetilde{\Delta }_{q-1}\left( D_{m}\right) $ for $q=0$ and $q=1$ as well
as the induction hypothesis (\ref{specDet}) to compute $\func{spec}%
\widetilde{\Delta }_{p-1}\left( D_{m}^{n}\right) $ for $p=r$ and $p=r-1$, we
obtain%
\begin{align}
\func{spec}\widetilde{\Delta }_{r-1}(D_{m}^{\left( n+1\right) })& =\left(
\func{spec}\widetilde{\Delta }_{r-1}\left( D_{m}^{n}\right) +\func{spec}%
\widetilde{\Delta }_{-1}\left( D_{m}\right) \right)  \notag \\
& \ \ \ \ \ \sqcup \left( \func{spec}\widetilde{\Delta }_{r-2}\left(
D_{m}^{n}\right) +\func{spec}\widetilde{\Delta }_{0}\left( D_{m}\right)
\right)  \notag \\
& =\left\{ \left( (n-k)m\right) _{\binom{r}{k}\binom{n}{r}%
(m-1)^{k}}+m\right\} _{k=0}^{r}  \label{3m} \\
& \ \ \ \ \ \sqcup \left\{ \left( (n-l)m\right) _{\binom{r-1}{l}\binom{n}{r-1%
}(m-1)^{l}}+\{0_{m-1},m\}\right\} _{l=0}^{r-1}.  \label{4m}
\end{align}%
The sequence in (\ref{3m}) is equal to
\begin{equation}
\left\{ \left( (n+1-k)m\right) _{\binom{r}{k}\binom{n}{r}(m-1)^{k}}\right\}
_{k=0}^{r},  \label{1m}
\end{equation}%
and the sequence in (\ref{4m}) is equal to
\begin{align}
& \left\{ \left( (n-l)m\right) _{\binom{r-1}{l}\binom{n}{r-1}%
(m-1)^{l+1}},\left( (n-l+1)m\right) _{\binom{r-1}{l}\binom{n}{r-1}%
(m-1)^{l}}\right\} _{l=0}^{r-1}\ \ \   \notag \\
& =\left\{ \left( (n-l)m\right) _{\binom{r-1}{l}\binom{n}{r-1}%
(m-1)^{l+1}}\right\} _{l=0}^{r-1}  \notag \\
& \ \ \ \ \ \ \ \sqcup \left\{ \left( (n+1-l)m\right) _{\binom{r-1}{l}\binom{%
n}{r-1}(m-1)^{l}}\right\} _{l=0}^{r-1}  \notag \\
& =\left\{ \left( (n+1-k)m\right) _{\binom{r-1}{k-1}\binom{n}{r-1}%
(m-1)^{k}}\right\} _{k=1}^{r}  \notag \\
& \ \ \ \ \ \ \ \sqcup \left\{ \left( (n+1-k)m\right) _{\binom{r-1}{k}\binom{%
n}{r-1}(m-1)^{k}}\right\} _{k=0}^{r-1}  \notag \\
& =\left\{ \left( (n+1-k)m\right) _{\binom{r-1}{k-1}\binom{n}{r-1}%
(m-1)^{k}}\right\} _{k=0}^{r}  \notag \\
& \ \ \ \ \ \ \ \sqcup \left\{ \left( (n+1-k)m\right) _{\binom{r-1}{k}\binom{%
n}{r-1}(m-1)^{k}}\right\} _{k=0}^{r}.\ \ \ \ \ \ \ \ \ \ \ \ \   \label{2m}
\end{align}%
Combining (\ref{1m}) and (\ref{2m}), we conclude that $\func{spec}\widetilde{%
\Delta }_{r-1}(D_{m}^{\left( n+1\right) })$ consists of the eigenvalues%
\begin{equation*}
\lambda _{k}=(n+1-k)m,\ \ k=0,...,r,
\end{equation*}%
where the multiplicity of $\lambda _{k}$ is%
\begin{eqnarray*}
&&\left[ \tbinom{r}{k}\tbinom{n}{r}+\tbinom{r-1}{k-1}\tbinom{n}{r-1}+\tbinom{%
r-1}{k}\tbinom{n}{r-1}\right] (m-1)^{k} \\
&&\ \ \underset{}{=}\left[ \tbinom{r}{k}\tbinom{n}{r}+\tbinom{r}{k}\tbinom{n%
}{r-1}\right] (m-1)^{k} \\
&&\ \ \underset{}{=}\tbinom{r}{k}\tbinom{n+1}{r}(m-1)^{k},
\end{eqnarray*}%
which proves the induction step.
\end{proof}

%TCIMACRO{%
%\TeXButton{Bibliography}{\addcontentsline{toc}{section}{References}
%{\footnotesize
%\input bibliog
%}}}%
%BeginExpansion
\addcontentsline{toc}{section}{References}
{\footnotesize
\input bibliog
}%
%EndExpansion

AG: Department of Mathematics, University of Bielefeld, 33501 Bielefeld,
Germany

\href{mailto:grigor@math.uni-bielefeld.de}{grigor@math.uni-bielefeld.de}

YL: Yau Mathematical Sciences Center and Department of Mathematics, Tsinghua
University, Beijing, 100084, China

\href{mailto:yonglin@tsinghua.edu.cn}{yonglin@tsinghua.edu.cn}

STY: Yau Mathematical Sciences Center and Department of Mathematics, Tsinghua
University, Beijing, 100084, China

\href{mailto:styau@tsinghua.edu.cn}{styau@tsinghua.edu.cn}

HZ: Yau Mathematical Sciences Center and Department of Mathematics, Tsinghua
University, Beijing, 100084, China

\href{mailto:zhanghh22@mails.tsinghua.edu.cn}{zhanghh22@mails.tsinghua.edu.cn%
}

\end{document}

%% file: tcirm.tex
\def\RM{\rm}

\def\limfunc#1{\mathop{\mathrm{#1}}}
\def\func#1{\mathop{\mathrm{#1}}\nolimits}

\def\Xint#1{\mathchoice
{\XXint\displaystyle\textstyle{#1}}%
{\XXint\textstyle\scriptstyle{#1}}%
{\XXint\scriptstyle\scriptscriptstyle{#1}}%
{\XXint\scriptscriptstyle\scriptscriptstyle{#1}}%
\!\int}
\def\XXint#1#2#3{{\setbox0=\hbox{$#1{#2#3}{\int}$ }
\vcenter{\hbox{$#2#3$ }}\kern-.6\wd0}}

\def\oint{\Xint-}

\def\endofall{\end{document}}

\def\endlecture#1#2#3{}

\def\enddoc{\end{document}}

\def\Qlb#1{#1}

\def\Qcb#1{#1} 

\def\FRAME#1#2#3#4#5#6#7#8
{
 \begin{figure}[H]
 \begin{center}
 \includegraphics[height=#3]{#7}
 \caption{#5}
 \label{#6}
 \end{center}
 \end{figure}
}